\documentclass[11pt,letterpaper]{article}
\usepackage{graphicx}

\usepackage[T1]{fontenc}
\usepackage[utf8]{inputenc}

\usepackage{hyperref}

\usepackage[margin=1.15in]{geometry}

\usepackage{amsmath, amsfonts,amssymb}
\usepackage{mathrsfs}

\usepackage{algorithm}
\usepackage{algorithmic}
\usepackage{caption}
\usepackage{mathtools}
 \usepackage{cite}

\newcommand{\cF}{{\mathscr F}}

\newcommand{\cO}{{\mathcal O}}

\DeclareMathOperator{\prox}{prox}
\DeclareMathOperator{\ran}{ran}

\DeclareMathOperator{\dom}{dom}
\newcommand{\eqdef}{\coloneqq}

\newcommand{\cX}{{\mathcal X}}
\newcommand{\cZ}{{\mathcal Z}}
\newcommand{\cY}{{\mathcal Y}}

\newcommand{\bR}{{\mathbb R}}
\newcommand{\bP}{{\mathbb P}}
\newcommand{\bE}{{\mathbb E}}

\newcommand{\cL}{{{\mathscr L}}}

\newcommand{\ps}[1]{\langle #1 \rangle}

\newcommand{\Span}{\mathop{\mathrm{span}}\nolimits}

\DeclareMathOperator*{\argmin}{arg\,min}
\DeclareMathOperator*{\minimize}{minimize}

\usepackage{color}
\definecolor{mydarkblue}{rgb}{0,0.0,0.8}
\definecolor{mydarkgreen}{rgb}{0,0.6,0.0}

\usepackage{amsthm}
\newtheorem{lemma}{Lemma}[section]
\newtheorem{theorem}{Theorem}[section]

\newtheorem{proposition}{Proposition}[section]
\newtheorem{assumption}{Assumption}
\newtheorem{remark}{Remark}[section]

\begin{document}

\title{\textbf{Dualize, Split, Randomize: \\Toward Fast Nonsmooth Optimization Algorithms}}

\author{Adil Salim, \ Laurent Condat,  \ Konstantin Mishchenko, \ Peter Richt\'arik\\
\phantom{xx}
 \\
 King Abdullah University of Science and Technology (KAUST)\\
 Thuwal 23955-6900, Kingdom of Saudi Arabia}

 \date{Authors' final version. 
Published in Journal of Optimization Theory and Applications, July 2022. \url{https://doi.org/10.1007/s10957-022-02061-8}}

\maketitle

\begin{abstract}
    We consider minimizing the sum of three convex functions, where the first one $F$ is smooth, the second one is nonsmooth and proximable and the third one is the composition of a nonsmooth proximable function with a linear operator $L$. This template problem has many applications, for instance, in image processing and machine learning. First, we propose a new primal--dual algorithm, which we call PDDY, for this problem. It is constructed by applying Davis--Yin splitting to a monotone inclusion in a primal--dual product space, where the operators are monotone under a specific metric depending on $L$. We show that three existing algorithms (the two forms of the Condat--V\~u algorithm and the PD3O algorithm) have the same structure, so that PDDY is the fourth missing link in this self-consistent class of primal--dual algorithms. This representation  eases the convergence analysis: it allows us to derive sublinear convergence rates in general, and linear convergence results in presence of strong convexity. Moreover, within our broad and flexible analysis framework, we propose new 
stochastic generalizations of the algorithms, in which a variance-reduced random estimate of the gradient of $F$ is used, instead of the true gradient. 
Furthermore, 
we obtain, as a special case of PDDY, a linearly converging algorithm for the minimization of a strongly convex function $F$ under a linear constraint; 
we discuss its important application to decentralized optimization.
\end{abstract}

\section{Introduction}
Many problems in statistics, machine learning or signal processing 
can be formulated as high-dimensional convex optimization problems~\cite{pal09, sta10, bac12, pol15, cha16, sta16}. They typically involve a smooth term $F$ and a nonsmooth regularization term $G$, and $F+G$ is often minimized using (a variant of) 
Stochastic Gradient Descent (SGD)~\cite{lan2020first}. 
However, in many cases, $G$ is not proximable; that is, its proximity operator does not admit a closed-form expression. 
In particular, structured regularization functions, like the total variation or its variants for images or graphs~\cite{rud92,bre10,cou13,con14,dur16,con17,sal-bia-hac-(sub)tac17}, 
or the overlapping group lasso~\cite{bac12}, are known to have computationally  expensive proximity operators. 
Also, when $G$ is a sum of several regularizers, $G$ is not proximable, even if the individual regularizers are, in general~\cite{pus17}. 
{Thus, in many situations,} $G$ is not proximable, but it takes the form $G = R + H \circ L$ where $R$, $H$ are proximable and $L$ is a linear operator. 
Therefore, in this paper, we study the problem 
\begin{equation}
    \label{eq:original-pb}
  \textbf{Problem (1)}:\ \ \   \minimize_{x\in\mathcal{X}}\, \Big(F(x)+R(x)+H(Lx)\Big),
\end{equation}
{where $L: \mathcal{X}\rightarrow \mathcal{Y}$  is a linear operator, $\mathcal{X}$ and $\mathcal{Y}$ are real Hilbert spaces (all spaces are supposed of finite dimension), $F:\mathcal{X}\rightarrow \mathbb{R}$ is a convex function, $R:\mathcal{X}\rightarrow \mathbb{R}\cup\{+\infty\}$ and $H:\mathcal{Y}\rightarrow \mathbb{R}\cup\{+\infty\}$ are proper, convex, lower semicontinuous functions; we refer to textbooks like \cite{boy04,bau17} for standard definitions of convex analysis. $F$ is supposed to be $\nu$-smooth, for some $\nu>0$; that is,  it is differentiable on $\mathcal{X}$ and its gradient $\nabla F$ is $\nu$-Lipschitz continuous: $\|\nabla F(x)-\nabla F(x')\|\leq \nu \|x-x'\| $, for every $(x,x') \in\mathcal{X}^2$.}

{Our contributions are the following. We recast Problem~\eqref{eq:original-pb} as finding a zero of the sum of three operators, which are monotone in a primal--dual product space, under a particular metric (Sect.~\ref{sec2}). Then, we apply Davis--Yin splitting (DYS)~\cite{dav17}, a generic method for this type of monotone inclusions (Sect.~\ref{sec3}). By doing so, we recover the existing PD3O~\cite{yan18} and two forms of the Condat--V\~u~\cite{con13,vu13} algorithms, but we also discover a new one, which we call the Primal--Dual Davis--Yin (PDDY) algorithm (Sect.~\ref{sec:pdalgos}). In other words, we discover PDDY as the fourth ``missing link'' in a group of primal--dual algorithms, which is self-consistent, in the sense that by exchanging the roles of the primal and dual terms $R+H\circ L$ and $R^* \circ (-L^*)  + H^*$, or by exchanging the roles of two monotone operators in the construction, we recover this or that algorithm. Furthermore, the decomposition of the primal--dual monotone inclusion into three terms allows us to use an important inequality regarding DYS for the analysis of the algorithms. More precisely, we can apply Lemma~\ref{lem:funda-DYS}, by instantiating the monotone operators and 
inner product with the ones at hand. Thanks to this property, we can  easily replace the gradient $\nabla F$ by a stochastic variance-reduced (VR) estimator, which can be much cheaper to evaluate (Sect.~\ref{sec:grad-estimators}). Thus, we derive the first VR stochastic algorithms to tackle Problem~\eqref{eq:original-pb}, to the best of our knowledge. We also leverage the DYS representation of the algorithms to prove convergence rates; our analysis covers the deterministic and stochastic cases in a unified way (Sect.~\ref{sec:grad-estimators}). Moreover, as a byproduct of our analysis, we discover the first
linearly converging algorithm for the minimization of a smooth strongly convex function, {using its gradient,} under a linear constraint, {without projections on it} (Sect.~\ref{sec7}). Its application to decentralized optimization is discussed in Appendix~\ref{apppridec}. 
Finally, numerical experiments illustrating the performance of the algorithms are presented in Sect.~\ref{sec:exp}. 
 A part of the proofs is deferred to Appendix~\ref{secappa} and additional linear convergence results are derived in Appendix~\ref{secappb}.
}

\subsection{Related Work} 

\textbf{Splitting algorithms}: Algorithms allowing to solve nonsmooth optimization problem involving several  proximable terms are called proximal splitting algorithms~\cite{com10,bot14,par14,kom15,bec17,con19,com21}. A classical one is the Douglas--Rachford algorithm~\cite{lio79, eck92,sva11,ryu20} (or, equivalently, the ADMM~\cite{glo75,gab76,boy11}) 
to minimize the sum of two nonsmooth functions $R+H$. To minimize $G = R+H \circ L$, the Douglas--Rachford algorithm can be generalized to the Primal--Dual Hybrid Gradient (PDHG) algorithm, a.k.a.\ Chambolle--Pock algorithm~\cite{cha11a,oco20}. Behind its success is the ability to handle the composite term $H \circ L$ using separate activation of $L$, its adjoint operator $L^T$, and the proximity operator of $H$.
 However, in many 
 applications, the objective function involves a smooth function 
 $F$, for instance, a least-squares term or a sum of logistic losses composed with inner products.  To cover these applications, proximal splitting algorithms like the {Combettes--Pesquet~\cite{com12}}, Condat--V\~u~\cite{con13,vu13} and PD3O~\cite{yan18} algorithms have been proposed; they can solve the general  Problem~\eqref{eq:original-pb}. These algorithms are primal--dual in nature; that is, they solve not only the 
primal problem \eqref{eq:original-pb}, but also  the dual problem, in a joint way. {Many other algorithms exist to solve Problem~\eqref{eq:original-pb}, and we refer to \cite{con19} and \cite{com21} for an overview of primal--dual proximal splitting algorithms. We can also mention the class of \emph{projective splitting} algorithms first proposed in \cite{eckps1} and further developed in several papers \cite{eckps2,eckps3,eckps4,eckps5,eckps6}. They proceed by  building  a separating hyperplane between the current iterate and the solution and then projecting the current iterate onto this hyperplane, to get closer to the solution. The projective splitting algorithms with forward steps \cite{eckps5,eckps6} are fully split and can solve Problem~\eqref{eq:original-pb}, as well.}\medskip

\noindent \textbf{Stochastic splitting algorithms}: In machine learning applications, the gradient of $F$ is often {much too expensive to evaluate} and replaced by a cheaper stochastic estimate. 
We can distinguish the two classes of 
standard stochastic gradients~\cite{gow19,
lan2020first,gor20} and 
variance-reduced (VR) stochastic gradients~\cite{joh13a,zha13,xia14,def14,gor20,gow20a}. 
VR stochastic gradients are estimators 
that ensure convergence to an exact solution of the problem, like with deterministic algorithms; that is, the variance of the stochastic errors they induce tends to zero. 
For some problems, VR stochastic algorithms are significantly faster than their deterministic counterparts. 
By contrast, with standard stochastic gradients {and constant stepsizes}, the algorithms typically do not converge to a fixed point and continue to fluctuate 
in a neighborhood of the solution set; this can be sufficient if the desired accuracy is low and speed is critical. 
When $L = I$, where $I$ denotes the identity, solving 
Problem~\eqref{eq:original-pb} with standard  and with VR stochastic gradients was considered in~\cite{yurtsever2016stochastic} and  in~\cite{pedregosa2019proximal}, respectively. In the general case $L \neq I$ of interest in this paper, solving the problem 
 with a standard stochastic gradient  was considered  in~\cite{zhao2018stochastic}. Thus, our proposed method is the first to allow solving the general Problem~\eqref{eq:original-pb} in a flexible way, with calls to $\nabla F$ or to standard or VR stochastic estimates.

\subsection{Mathematical Background}\label{secmb}

We introduce some notions and notations of convex analysis and operator theory, see \cite{boy04,bau17} 
 for more details. 
  Let $\mathcal{Z}$ be a real Hilbert space.  Let $G:\mathcal{Z}\rightarrow\mathbb{R}\cup\{+\infty\}$ be a convex function. 
 The domain of $G$ is the convex set  $\dom(G)=\{z\in\mathcal{Z}\;:\;G(z)\neq +\infty\}$, 
 its subdifferential 
 is the set-valued operator $\partial G:z\in\mathcal{Z}\mapsto \{y\in\mathcal{Z}\ :\ (\forall z'\in\mathcal{Z})\ G(z)+\langle z'-z, y\rangle\leq G(z')\}$, and its conjugate function is $G^*:z\mapsto \sup_{z'\in\mathcal{Z}} \{\langle z,z'\rangle -G(z')\}$. 
 If $G$ is differentiable at $z\in\mathcal{Z}$, $\partial G(z)=\{\nabla G(z)\}$. 
   We define the proximity operator of $G$ as the 
  operator
  $\mathrm{prox}_G:z\in\mathcal{Z}\mapsto \argmin_{z'\in\mathcal{Z}}\big\{G(z')+{\textstyle\frac{1}{2}}\|z-z'\|^2\big\}$. Finally, given any $b\in\mathcal{Z}$, we define the convex indicator function $\iota_b:z\mapsto \{0$ if $z=b$, $+\infty$ otherwise$\}$. 

Let $M : \cZ \rightarrow 2^\cZ$ be a set-valued operator. The inverse $M^{-1}$ of $M$ is defined by the relation $z' \in M(z) \Leftrightarrow z \in M^{-1}(z')$. The set of zeros of $M$ is $\mathrm{zer}(M) \eqdef 
\{z \in \cZ, 0 \in M(z)\}$.  $M$ is monotone if $\ps{w-w',z-z'} \geq 0$ and strongly monotone if there exists $\mu >0$ such that $\ps{w-w',z-z'} \geq \mu\|z-z'\|^2$, for every $(x,x')\in\mathcal{Z}^2$, $w \in M(z)$, $w' \in M(z')$. $M$ is maximally monotone if its graph is not contained in the graph of another monotone operator.
The resolvent of $M$ is $J_{M} \eqdef (I + M)^{-1}$. 
  If $G$ is proper, convex and lower semicontinuous, 
  $\partial G$ is maximally monotone, $J_{\partial G} = \prox_{G}$, $\mathrm{zer}(\partial G) = \argmin G$ and $(\partial G)^{-1} = \partial G^*$.

A single-valued operator $M$ on $\cZ$ is $\xi$-cocoercive if $\xi\|M(z) - M(z')\|^2 \leq \ps{M(z) - M(z'),z-z'}$, for every $(z,z')\in \mathcal{Z}^2$.  
 The resolvent of a maximally monotone operator is $1$-cocoercive and $\nabla G$ is $1/\nu$-cocoercive, for any $\nu$-smooth function $G$.

The adjoint of a linear operator  $P$ is denoted by $P^*$ and its operator norm by $\|P\|$. $P$ is self-adjoint if $P=P^*$.   
Let $P : \cZ \to \cZ$ be a self-adjoint linear operator. $P$ is positive  if $\ps{Pz,z} \geq 0$, for every $z \in \cZ$,  and strongly positive  if, additionally, $\ps{Pz,z} = 0$ implies $z=0$. In this latter case, the inner product induced by $P$ is defined by $\ps{z,z'}_P \eqdef \ps{Pz,z'}$ and the norm induced by $P$ by $\|z\|_P \eqdef \ps{z,z}_P^{1/2}$. We denote by $\mathcal{Z}_P$ the real Hilbert space made from the vector space $\mathcal{Z}$ endowed with $\ps{\cdot,\cdot}_P$.

\section{Primal--Dual Formulation and Optimality Conditions}\label{sec2}

For Problem~\eqref{eq:original-pb} to be well posed, we suppose that there exists $x^\star\in\mathcal{X}$, such that 
\begin{equation}
0\in \nabla F(x^\star) +\partial R(x^\star) +L^* \partial H(Lx^\star).\label{eqc0}
\end{equation}
 Then, $x^\star$ is a  solution to~\eqref{eq:original-pb}. For instance, a standard qualification constraint for this assumption to hold is that $0$ belongs to the relative interior of $\dom(H) - L\dom(R)$~\cite{com12}. Then, for every $x^\star$ satisfying \eqref{eqc0}, there exists $y^\star \in \partial H (L x^\star)$ such that $0\in \nabla F(x^\star) +\partial R(x^\star) +L^* y^\star$; equivalently, 
 $(x^\star,y^\star)$ is a zero of the set-valued operator $M$ defined by
\begin{equation}
    \label{eq:M}
    M : (x,y) \in\mathcal{X}\times\mathcal{Y} \mapsto  \Big(\nabla F(x)+\partial R(x)+L^* y, -Lx +\partial H^*(y)\Big).
\end{equation}
Conversely, for every $(x^\star,y^\star) \in \mathrm{zer}(M)$, $x^\star$ is a solution to \eqref{eq:original-pb} and $y^*$ is a solution to the dual problem 
\begin{equation}
\minimize_{y \in \mathcal{Y}} \,\Big((F+R)^*(-L^* y) + H^*(y)\Big),\label{eqdual}
\end{equation}
see Sect.~15.3 of~\cite{bau17}; 
moreover, there exist
$r^\star \in \partial R(x^\star)$ and $h^\star \in \partial H^*(y^\star)$ such that, using 2-block vector notations in $\mathcal{X}\times\mathcal{Y}$,
\begin{equation}
    \label{eq:saddle0}
    \begin{bmatrix} 0 \\ 0\end{bmatrix} = \begin{bmatrix}  \nabla F(x^\star)+ r^\star+L^* y^\star \\ -L x^\star+h^\star\end{bmatrix}.
\end{equation}
In the sequel, we let $(x^\star,y^\star) \in \mathrm{zer}(M)$ and $r^\star,h^\star$ be any elements such that Eqn.~\eqref{eq:saddle0} holds.

A zero of $M$ is also a saddle point of the convex--concave Lagrangian function, defined as
\begin{equation}
\label{eq:lagrangian}
    \cL(x,y) \coloneqq F(x)+R(x) - H^*(y) + \ps{Lx,y}.
\end{equation}
For every $x \in \mathcal{X}$ and $y \in \mathcal{Y}$, we define the \emph{Lagrangian gap} at $(x,y)$ as $\cL(x,y^\star) - \cL(x^\star,y)$. The following holds:
\begin{lemma}[Lagrangian gap]
    \label{lem:duality-gap}
    For every $x \in \mathcal{X}$ and $y \in \mathcal{Y}$, we have
    \begin{equation}
        \cL(x,y^\star) - \cL(x^\star,y) = D_F(x,x^\star)+D_R(x,x^\star)+D_{H^*}(y,y^\star),
    \end{equation}
    where the Bregman divergence of the smooth function $F$ between any two points $x$ and $x'$  is $D_F(x,x') \coloneqq F(x) - F(x') - \ps{\nabla F(x'),x-x'}$, and
          $D_R(x,x^\star) \eqdef R(x) - R(x^\star) - \ps{r^\star,x-x^\star}$, 
        $D_{H^{*}}(y,y^\star) \eqdef H^*(y) - H^*(y^\star) - \ps{h^\star,y-y^\star}$.
    \end{lemma}

\begin{proof}
Using the optimality conditions~\eqref{eq:saddle0}, we have
\begin{align*}
    D_F(x,x^\star)+D_R(x,x^\star) &= (F+R)(x) - (F+R)(x^\star) -\ps{\nabla F(x^\star) + r^\star,x-x^\star}\notag\\
    &= (F+R)(x) - (F+R)(x^\star) +\ps{L^* y^\star,x-x^\star}\\
    &= (F+R)(x) - (F+R)(x^\star) +\ps{y^\star,L x} - \ps{y^\star, L x^\star}.\notag
\end{align*}
We also have
\begin{align*}
    D_{H^*}(y,y^\star) 
    &= H^*(y) - H^*(y^\star) -\ps{L x^\star,y-y^\star}\\
    &= H^*(y) - H^*(y^\star) -\ps{L x^\star,y} +\ps{y^\star,L x^\star}.\notag
\end{align*}
Hence, 
\begin{align*}
&D_F(x,x^\star)+D_R(x,x^\star)+D_{H^*}(y,y^\star)\notag\\
{}={}& (F+R)(x) - (F+R)(x^\star) + H^*(y) - H^*(y^\star) -\ps{L x^\star,y} +\ps{y^\star,L x}\\
{}={}& \cL(x,y^\star) - \cL(x^\star,y).
\end{align*}\end{proof}

For every $x \in \mathcal{X}$, $y \in \mathcal{Y}$, Lemma~\ref{lem:duality-gap} and the convexity of $F,R,H^*$ imply that 
$ 
  \cL(x^\star,y) \leq \cL(x^\star,y^\star) \leq \cL(x,y^\star)
$. 
 So, the Lagrangian gap $\cL(x,y^\star) - \cL(x^\star,y)$ is nonnegative, and it is zero 
if  $x$ is a solution to Problem~\eqref{eq:original-pb} and $y$ is a solution to
the dual problem \eqref{eqdual}. The converse is not always true, generally speaking. But in realistic situations, this is the case, and under mild assumptions, like strict convexity of the functions around $x^\star$ and $y^\star$, the Lagrangian gap converging to zero is a valid measure of convergence to a solution.

The operator $M$ defined in~\eqref{eq:M} can be shown to be maximally monotone. Moreover, we have
\begin{align}
    M(x,y) &= \begin{bmatrix} \partial R(x) \\ 0 \end{bmatrix} +  \begin{bmatrix} &   L^* y \\ -L x\!\!& {}+\partial H^*(y)\end{bmatrix}+\begin{bmatrix} \nabla F(x)\\ 0\end{bmatrix} \label{eq:saddle}\\  
    &= \begin{bmatrix} 0 \\ \partial H^*(y) \end{bmatrix} +  \begin{bmatrix} \partial R(x)\!\!&{}+L^* y \\ -L x\end{bmatrix}+\begin{bmatrix} \nabla F(x)\\ 0\end{bmatrix}, \label{eq:saddle12}
    \end{align}
and each term at the right hand side of~\eqref{eq:saddle} or~\eqref{eq:saddle12} is maximally monotone, see Corollary 25.5 in~\cite{bau17}.

\begin{figure*}[t]
\begin{minipage}{.48\textwidth}\begin{algorithm}[H]
 \caption*{\textbf{Davis--Yin Splitting alg.} \\DYS$(\tilde{A},\tilde{B},\tilde{C})$~\cite{dav17}}
 \begin{algorithmic}[1]
    \STATE \textbf{Input:} $v^0\in \mathcal{Z}$, $\gamma>0$
 	\FOR{$k = 0,1,2,\dots$}
  \STATE $z^{k} = J_{\gamma \tilde{B}}(v^k)$
  \STATE $u^{k+1} = J_{\gamma \tilde{A}}(2 z^{k} - v^k - \gamma \tilde{C}(z^{k}))$
  \STATE $v^{k+1} = v^k + u^{k+1} - z^{k}$
 	\ENDFOR
 \end{algorithmic}
 \end{algorithm}
 \end{minipage}\ \ \ \ \   \begin{minipage}{.48\textwidth}\begin{algorithm}[H]
\caption*{\textbf{LiCoSGD} (new)}
 \begin{algorithmic}[1]
    \STATE \textbf{Input:} $x^0 \in \cX, y^0 \in \cY $, $\gamma>0$, $\tau>0$
 	\FOR{$k = 0,1,2,\dots$}
 	\STATE $w^k = x^k -\gamma g^{k+1}$
	\STATE $y^{k+1}=y^k + \tau L (w^k-\gamma L^* y^k)-\tau b$%
	\STATE $x^{k+1} =w^k - \gamma L^* y^{k+1}$
 	\ENDFOR
 \end{algorithmic}\end{algorithm}\end{minipage}
\medskip

{ \normalsize Note : the deterministic versions of the algorithms are obtained by setting $g^{k+1}=\nabla F(x^k)$.}
 
\begin{minipage}{.48\textwidth}\begin{algorithm}[H]
\caption*{\textbf{Stochastic PDDY alg.} (new)}
 \begin{algorithmic}[1]
    \STATE \textbf{Input:} $p^0 \in \cX, y^0 \in \cY$, $\gamma>0$, $\tau>0$
 	\FOR{$k = 0,1,2,\dots$}
 	\STATE $y^{k+1} \!= \!\prox_{\tau  H^*\!}\!\big(y^{k}+\tau L (p^k-\gamma L^*y^{k})\big)$
 	\STATE $x^k=p^k -\gamma L^* y^{k+1}$
	\STATE $s^{k+1}=\prox_{\gamma R}\big(2x^k-p^k-\gamma g^{k+1}\big)$%
	\STATE $p^{k+1}=p^k+s^{k+1}-x^k$
 	\ENDFOR
 \end{algorithmic}
 \end{algorithm}\end{minipage}\ \ \ \ \ \ \begin{minipage}{.48\textwidth}\begin{algorithm}[H]
\caption*{\textbf{Stochastic PD3O alg.} (new)}
 \begin{algorithmic}[1]
    \STATE \textbf{Input:} $p^0 \in \cX, y^0 \in \cY $, $\gamma>0$, $\tau>0$
 	\FOR{$k = 0,1,2,\dots$}
 	\STATE $x^k=\prox_{\gamma R}(p^k)$
 	\STATE $w^k = 2x^k - p^k-\gamma g^{k+1}$
	\STATE $\!y^{k+1} \!= \!\prox_{\tau H^{\!*}\!}\!\big( y^k + \tau L (w^k-\gamma L^* y^k)\big)$
	\STATE $p^{k+1} = x^k - \gamma g^{k+1} - \gamma L^* y^{k+1}$
 	\ENDFOR
 \end{algorithmic}
\end{algorithm}\end{minipage}
 \end{figure*}

\section{Davis--Yin Splitting}\label{sec3}
Solving Problem~\eqref{eq:original-pb} boils down to finding a zero $(x^\star,y^\star)$ of the monotone operator $M$ defined in \eqref{eq:M}, which can be written as the sum of three monotone operators, like in \eqref{eq:saddle} or \eqref{eq:saddle12}. The method proposed by Davis and Yin~\cite{dav17}, which we call Davis--Yin splitting (DYS), 
 is dedicated to this problem; that is, find a zero of the sum of three monotone operators, one of which is cocoercive.

Let $\mathcal{Z}$ be a real Hilbert space. Let $\tilde{A}, \tilde{B}, \tilde{C}$  be maximally monotone operators on $\mathcal{Z}$. We assume that $\tilde{C}$ is $\xi$-cocoercive, for some $\xi >0$.
The DYS algorithm, denoted by $\text{DYS}(\tilde{A},\tilde{B},\tilde{C})$ and shown above, aims at finding an element in  $\mathrm{zer}(\tilde{A}+\tilde{B}+\tilde{C})$, supposed nonempty.
The fixed points of $\text{DYS}(\tilde{A},\tilde{B},\tilde{C})$ are the triplets $(v^\star, z^\star, u^\star)\in \mathcal{Z}^3$, such that
\begin{equation}
\label{eq:DYSfix}
z^\star = J_{\gamma \tilde{B}}(v^\star),\quad u^\star = J_{\gamma \tilde{A}}\big(2 z^\star - v^\star - \gamma \tilde{C}(z^\star)\big),\quad u^\star = z^\star.
\end{equation}
These fixed points are related to the zeros of $\tilde{A}+\tilde{B}+\tilde{C}$ as follows, see Lemma~2.2 in ~\cite{dav17}: for every $(v^\star, z^\star, u^\star)\in \mathcal{Z}^3$ satisfying \eqref{eq:DYSfix}, $z^\star \in \mathrm{zer}(\tilde{A}+\tilde{B}+\tilde{C})$. Conversely, for every $z^\star \in \mathrm{zer}(\tilde{A}+\tilde{B}+\tilde{C})$, there exists $(v^\star, u^\star) \in\mathcal{Z}^2$, such that  $(v^\star, z^\star, u^\star)$ satisfies  \eqref{eq:DYSfix}. We have~\cite{dav17}:

\begin{lemma}[Convergence of the DYS algorithm]
    \label{lem:DYS-cv}
    Suppose that $\gamma\in (0,2\xi)$. Then the sequences  $(v^k)_{k\in\mathbb{N}}$, $(z^k)_{k\in\mathbb{N}}$, $(u^k)_{k\in\mathbb{N}}$ generated by $\text{DYS}(\tilde{A},\tilde{B},\tilde{C})$ 
converge 
to some elements $v^\star$, $z^\star$, $u^\star$ in $\mathcal{Z}$, respectively. Moreover, $(v^\star, z^\star, u^\star)$ satisfies  \eqref{eq:DYSfix} and $u^\star = z^\star \in \mathrm{zer}(\tilde{A}+\tilde{B}+\tilde{C})$.
\end{lemma}
The following equality is at the heart of the convergence proofs: 
\begin{lemma}[Fundamental equality of the DYS algorithm]
\label{lem:funda-DYS}
Let $(v^k,z^k,$ $u^k) \in \cZ^3$ be the iterates of the DYS algorithm, and $(v^\star, z^\star, u^\star)\in \cZ^3$ be such that~\eqref{eq:DYSfix} holds.
Then, for every $k \geq 0$, there exist $b^k \in \tilde{B}(z^k)$, $b^\star \in \tilde{B}(z^\star)$, $a^{k+1} \in \tilde{A}(u^{k+1})$ and $a^\star \in \tilde{A}(u^\star)$, such that
\begin{align}
    \label{eq:funda}
    \|v^{k+1} -v^\star\|^2 &= \|v^k - v^\star\|^2 -2\gamma\ps{b^k - b^\star,z^{k} - z^\star}-2\gamma\ps{a^{k+1} - a^\star,u^{k+1} - u^\star}\notag\\
   & \quad-2\gamma\ps{\tilde{C}(z^{k}) - \tilde{C}(z^\star),z^{k} - z^\star}+\gamma^2\|\tilde{C}(z^{k}) - \tilde{C}(z^\star)\|^2\\
    &\quad -\gamma^2\|a^{k+1}+b^k - a^{\star}-b^\star\|^2\notag.
\end{align}
\end{lemma}
\begin{proof}
Since $z^k = J_{\gamma \tilde{B}}(v^k)$, $z^k \in v^k - \gamma \tilde{B}(z^k)$ by definition of the resolvent. Therefore, there exists $b^k \in \tilde{B}(z^k)$, such that $z^k = v^k - \gamma b^k$. Similarly, 
$u^{k+1} \in 2z^k - v^k - \gamma \tilde{C}(z^{k}) - \gamma \tilde{A}(u^{k+1}) = v^k - 2\gamma b^k - \gamma \tilde{C}(z^{k}) - \gamma \tilde{A}(u^{k+1})$. 
 Therefore, there exists $a^{k+1} \in \tilde{A}(u^{k+1})$, such that
\begin{equation}
    \label{eq:DYS2}
     \left\{\begin{array}{l}
        z^{k} = v^k - \gamma b^k \\
        u^{k+1} = v^k - 2 \gamma b^k - \gamma \tilde{C}(z^{k}) - \gamma a^{k+1} \\
        v^{k+1} = v^k + u^{k+1} - z^{k}.
        \end{array}\right.
\end{equation}
Moreover,
$
    v^{k+1} = v^{k} - \gamma b^k - \gamma \tilde{C}(z^{k}) - \gamma a^{k+1}
$. 
Similarly, there exist $a^\star \in \tilde{A}(u^\star)$ and $b^\star \in \tilde{B}(z^\star)$, such that
\begin{equation}
    \label{eq:DYS2-fix}
    \left\{\begin{array}{l}
        z^\star = v^\star - \gamma b^\star \\
        u^\star = v^\star - 2 \gamma b^\star - \gamma \tilde{C}(z^\star) - \gamma a^\star \\
        v^\star = v^\star + u^\star - z^\star,
        \end{array}\right.
\end{equation}
and
   $ v^\star = v^\star - \gamma b^\star - \gamma \tilde{C}(z^\star) - \gamma a^\star
$. 
Therefore, 
\begin{align*}
    \|v^{k+1} - v^\star\|^2 ={}& \|v^k - v^\star\|^2 -2\gamma\Big\langle a^{k+1}+b^k + \tilde{C}(z^{k}) - \big(a^{\star}+b^\star + \tilde{C}(z^{\star})\big),\\
    &v^{k} - v^\star\Big\rangle+\gamma^2\big\|a^{k+1}+b^k + \tilde{C}(z^{k}) - \big(a^{\star}+b^\star + \tilde{C}(z^{\star})\big)\big\|^2.\notag
\end{align*}
By expanding the last squared norm  and by using~\eqref{eq:DYS2} and~\eqref{eq:DYS2-fix} in the inner product, we get
\begin{align*}
    \|v^{k+1} - &v^\star\|^2 = \|v^k - v^\star\|^2-2\gamma\ps{b^k + \tilde{C}(z^{k})- \big(b^\star + \tilde{C}(z^\star)\big),z^{k} - z^\star}\notag\\
    &-2\gamma\ps{a^{k+1}- a^\star,u^{k+1} - u^\star}\notag\\
    &-2\gamma\ps{b^k + \tilde{C}(z^{k})- \big(b^\star + \tilde{C}(z^\star)\big),\gamma b^k - \gamma b^\star}\notag\\
    &-2\gamma\ps{a^{k+1}- a^\star,2 \gamma b^k + \gamma \tilde{C}(z^{k}) + \gamma a^{k+1} - \big(2 \gamma b^\star + \gamma \tilde{C}(z^{\star}) + \gamma a^\star\big)}\\
    &+\gamma^2\|a^{k+1}+b^k- \big(a^\star+b^\star\big)\|^2+\gamma^2\|\tilde{C}(z^{k})- \tilde{C}(z^\star)\|^2\notag\\
    &+2\gamma^2\ps{a^{k+1}+b^k - \big(a^\star+b^\star\big),\tilde{C}(z^{k})- \tilde{C}(z^\star)}\\
    &\hspace{-7mm}= \|v^k - v^\star\|^2-2\gamma\ps{b^k -b^\star ,z^{k} - z^\star}-2\gamma\ps{a^{k+1}- a^\star,u^{k+1} - u^\star}\notag\\
    &-2\gamma\ps{ \tilde{C}(z^{k})-  \tilde{C}(z^\star),z^{k} - z^\star}+\gamma^2\|\tilde{C}(z^{k})- \tilde{C}(z^\star)\|^2\\
   &-2\gamma^2 \|b^k -b^\star\|^2
    -2\gamma^2\ps{a^{k+1}- a^\star,2  b^k  + a^{k+1} - 2  b^\star  - a^\star}\\
    &+\gamma^2\|a^{k+1}+b^k- \big(a^\star+b^\star\big)\|^2\notag
\end{align*}
{After combining the last three terms into $-\gamma^2\|a^{k+1}+b^k- \big(a^\star+b^\star\big)\|^2$, we obtain the result.}
\end{proof}

\section{A Class of Four Primal--Dual Optimization Algorithms}\label{sec:pdalgos}

We now set $\mathcal{Z}\coloneqq\mathcal{X}\times\mathcal{Y}$, where $\mathcal{X}$ and $\mathcal{Y}$ are the spaces defined in Sect.~\ref{sec2}. To solve the primal--dual problem \eqref{eq:saddle} or \eqref{eq:saddle12}, which consists in finding a zero of the sum $A+B+C$ of three operators in $\mathcal{Z}$, of which $C$ is cocoercive, a natural idea is to apply the Davis--Yin algorithm  DYS$(A,B,C)$. But the resolvent of $A$ or $B$ is often intractable. In this section, we show that preconditioning is the solution; that is,  we exhibit a strongly positive linear operator $P$, such that DYS$(P^{-1}A,P^{-1}B,P^{-1}C)$  is tractable. Since
$P^{-1}A,P^{-1}B,P^{-1}C$ are monotone operators in $\mathcal{Z}_P$,
 the algorithm will converge to a  zero of $P^{-1}A + P^{-1}B + P^{-1}C$, or, equivalently, of $A+B+C$.
Let us apply this idea in four different ways.

\subsection{A New Primal--Dual Algorithm: The PDDY Algorithm}

Let $\gamma >0$ and $\tau>0$ be real parameters. We introduce the four operators on $\mathcal{Z}$, using matrix-vector notations:
  \begin{align}
    A(x,y)=\begin{bmatrix} &   L^* y \\ -L x\!\!&{}+\partial H^*(y)\end{bmatrix},\ &B(x,y)=\begin{bmatrix} \partial R(x) \\ 0 \end{bmatrix}\!,\ C(x,y)=\begin{bmatrix} \nabla F(x)\\ 0\end{bmatrix},\notag\\
    &P=\begin{bmatrix} I&  0 \\ 0 & \ \ \frac{\gamma}{\tau}I - \gamma^2 L L^* \end{bmatrix}\label{eq:ABCP}.
  \end{align}
$P$ is strongly positive  if and only if $\gamma\tau \|L\|^2 < 1$. Since $A$, $B$, $C$ are maximally monotone in $\mathcal{Z}$, $P^{-1}A,P^{-1}B,P^{-1}C$ are maximally monotone in $\mathcal{Z}_P$. Moreover, $P^{-1}C$ is $1/\nu$-cocoercive in $\mathcal{Z}_P$. Importantly, we have:
\begin{align}
    P^{-1}C:(x,y) \mapsto \big(\nabla F(x),0\big),\ \ &\ \  J_{\gamma P^{-1}B}:(x,y)\mapsto \big(\mathrm{prox}_{\gamma R}(x),y\big)\label{algorithm_reslv},\\
    J_{\gamma P^{-1}A}:(x,y)\mapsto (x',y'),\ &\mbox{ where } 
    \left\lfloor
       \begin{array}{l}
       y'
   = \mathrm{prox}_{\tau H^*}\big(y+\tau L (x - \gamma L^* y)\big)\\
   x'=x -\gamma L^* y'.
   \end{array}\right.\notag
\end{align}
The form of the last resolvent was shown in \cite{oco20}; see also \cite{con19}, where this resolvent appears as one iteration of the Proximal Method of Multipliers. 
We plug these explicit steps into the Davis--Yin algorithm DYS$(P^{-1}B,P^{-1}A,$ $P^{-1}C)$ and we identify the variables as  $v^k = (p^k,q^k)$, $z^k = (x^k,y^{k+1})$, $u^k = (s^k,d^k)$, for some variables $(p^k, x^k, s^k) \in \cX^3$ and $(q^k, y^k, d^k) \in \cY^3$. Thus, we do the following substitutions:

$\bullet\ \ $Using~\eqref{algorithm_reslv}, the step
$
    z^k = J_{\gamma P^{-1}A}(v^k),
$
is equivalent to
\begin{equation*}
\left\lfloor\begin{array}{l}
    y^{k+1} = \prox_{\tau  H^*}\big((I-\tau\gamma LL^*)q^{k} +\tau L p^k \big)\\
    x^k = p^k -\gamma L^* y^{k+1}\notag
    \end{array}\right.
\end{equation*}

$\bullet\ \ $The step
$
    u^{k+1} = J_{\gamma P^{-1}B}\big(2z^k - v^k - \gamma P^{-1}C(z^k)\big)
$
is equivalent to
\begin{equation*}\left\lfloor\begin{array}{l}
    s^{k+1} =\prox_{\gamma R}\big(2x^k-p^k-\gamma \nabla F(x^k)\big)\\
    d^{k+1} = 2y^{k+1} - q^k.
\end{array}\right.\end{equation*}

$\bullet\ \ $Finally, the step
$
    v^{k+1} = v^k + u^{k+1} - z^k
$
is equivalent to
\begin{equation*}\left\lfloor\begin{array}{l}
    p^{k+1} = p^k + s^{k+1} - x^k\\
    q^{k+1} = q^k + d^{k+1} - y^{k+1}.
\end{array}\right.\end{equation*}
We can replace $q^k$ by $y^k$ and discard $d^k$, which is not needed. 
This yields the new Primal--Dual Davis--Yin (PDDY) algorithm, shown above (with $g^{k+1}=\nabla F(x^k)$). 
Note that it can be written with only one call to $L$ and $L^*$ per iteration.
Also, the PDDY Algorithm could be overrelaxed~\cite{con19}, since this possibility exists for the Davis--Yin algorithm. We have:
\begin{theorem}[Convergence of the PDDY Algorithm]
    \label{th:pddy-cv}
    Suppose that $\gamma\in (0,2/\nu)$ and that $\tau\gamma\|L\|^2<1$. Then the sequences $(x^k)_{k\in\mathbb{N}}$ and $(s^k)_{k\in\mathbb{N}}$  generated by the PDDY Algorithm converge to the same  
    solution $x^\star$ to Problem~\eqref{eq:original-pb}, and the sequence $(y^k)_{k\in\mathbb{N}}$ converges to some dual solution $y^\star$ of \eqref{eqdual}.
\end{theorem}
\begin{proof}
    Under the assumptions of Theorem~\ref{th:pddy-cv}, $P$ is strongly positive. Then the result follows from Lemma~\ref{lem:DYS-cv} applied in $\mathcal{Z}_P$ and from the analysis in Sect.~\ref{sec2}.
\end{proof}

\subsection{The PD3O Algorithm}
We consider the same notations as in the previous section. We switch the roles of $A$ and $B$ and consider DYS$(P^{-1}A,P^{-1}B,P^{-1}C)$. Then, after some substitutions similar to the ones done to construct the PDDY algorithm, 
 we recover exactly the PD3O algorithm proposed in~\cite{yan18}. 
Although it is not derived this way, its interpretation as a primal--dual Davis--Yin algorithm is mentioned by its author. Its convergence properties are the same as for the PDDY Algorithm, as stated in Theorem~\ref{th:pddy-cv}. 

In a recent work~\cite{oco20}, the PD3O algorithm has been shown to be an instance of the Davis--Yin algorithm, with a different reformulation, which does not involve duality. The authors of the present paper developed this technique further, applied it to the PDDY algorithm as well, and obtained convergence rates and accelerations for both algorithms~\cite{con20}.

\subsection{The Condat--V\~u Algorithm}\label{seccv}

    Let $\gamma >0$ and $\tau>0$ be real parameters. We want to study the decomposition \eqref{eq:saddle12} instead of \eqref{eq:saddle}. For this, we define the operators
       \begin{equation}
        \label{eq:ABC2}
        \bar{A}(x,y)=\begin{bmatrix}\partial R(x)&{}  + L^* y \\ -L x&\end{bmatrix}\!,\ \bar{B}(x,y)=\begin{bmatrix}  0\\ \partial H^*(y)\end{bmatrix}\!,\      
        Q=\begin{bmatrix} K\ &  \ 0 \\ 0\  & \ I \end{bmatrix},
      \end{equation}
where $K \eqdef \frac{\gamma}{\tau}I - \gamma^2 L^* L$, and we define $C$ like in \eqref{eq:ABCP}.  If $\gamma\tau \|L\|^2 < 1$, $K$ and $Q$ are strongly positive. In that case, since $\bar{A}$, $\bar{B}$, $C$ are maximally monotone in $\mathcal{Z}=\mathcal{X}\times\mathcal{Y}$, $Q^{-1}\bar{A}$, $Q^{-1}\bar{B}$, $Q^{-1}C$ are maximally monotone in $\mathcal{Z}_Q$.
Moreover, we have:
\begin{align}
    Q^{-1}C:(x,y)\mapsto \big(K^{-1}\nabla F(x)&,0\big)\label{eq:QC},\ \ \  J_{\gamma Q^{-1}\bar{B}}:(x,y)\mapsto \big(x,\mathrm{prox}_{\gamma H^*}(y)\big),\\
     J_{\gamma Q^{-1}\bar{A}}:(x,y)\mapsto (x',y'),\ &\mbox{ where } 
 \left\lfloor
    \begin{array}{l}
    x'
= \mathrm{prox}_{\tau R}\big((I-\tau\gamma L^* L) x-\tau L^* y \big)\\
y'=y +\gamma L x'.
\end{array}\right.\notag
\end{align}

If we plug these explicit steps into the Davis--Yin algorithm DYS$(Q^{-1}\bar{A},$ $Q^{-1}\bar{B},Q^{-1}C)$ or DYS$(Q^{-1}\bar{B},Q^{-1}\bar{A},Q^{-1}C)$, and after straightforward simplifications, we recover the two forms of the Condat--V\~u algorithm~\cite{con13,vu13}; that is, Algorithms 3.1 and  3.2 of~\cite{con13}, respectively, see also in \cite{con19}. 
The Condat--V\~u algorithm has the form of a  primal--dual forward--backward algorithm~\cite{com14,kom15,con19}. We have just shown that it can be viewed as a primal--dual Davis--Yin algorithm, with a different metric, as well. Furthermore, it is easy to show that 
$Q^{-1}C$ is $\xi$-cocoercive, with $\xi = (\frac{\gamma}{\tau} - \gamma^2\|L\|^2)/\nu$.
Hence, convergence follows from Lemma~\ref{lem:DYS-cv}, under the same condition on $\tau$ and $\gamma$ as in Theorem 3.1 of~\cite{con13}, namely $\frac{\nu}{2} < \frac{1}{\tau} - \gamma\|L\|^2$.

\section{Stochastic Primal--Dual Algorithms}\label{sec:grad-estimators}

We now introduce stochastic versions of the PD3O and PDDY algorithms; we omit the analysis of stochastic versions of the Condat--V\~u algorithm, which is the same,  with added technicalities due to cocoercivity with respect to the metric induced by $Q$ in \eqref{eq:ABC2}. 
Our approach has a `plug-and-play' flavor: we show that we have all the ingredients to leverage the unified theory of stochastic gradient estimators recently presented in \cite{gor20}. 

In the stochastic versions of the algorithms, the gradient $\nabla F(x^k)$ is replaced by a stochastic gradient $g^{k+1}$. That is, we consider a filtered probability space $(\Omega,\cF,(\cF_k)_{k\in\mathbb{N}},\bP)$, an $(\cF_k)$-adapted stochastic process $(g^k)_{k\in\mathbb{N}}$, we denote by $\bE$ the expectation and by $\bE_k$ the conditional expectation w.r.t. $\cF_k$. The following assumption is made on the process $(g^k)_{k\in\mathbb{N}}$.
\begin{assumption}
    \label{as:sto-grad}
    There exist $\alpha,\beta,\delta \geq 0$, $\rho \in (0,1]$ and a $(\cF_k)_{k\in\mathbb{N}}$-adapted stochastic process denoted by $(\sigma_k)_{k\in\mathbb{N}}$, such that, for every $k \in \mathbb{N}$, $\bE_k(g^{k+1}) = \nabla F(x^k)$, $\ \bE_k(\|g^{k+1} - \nabla F(x^\star)\|^2) \leq 2\alpha D_F(x^k,x^\star) + \beta\sigma_k^2\ $, and $\ \bE_k(\sigma_{k+1}^2) \leq (1-\rho)\sigma_k^2 + 2\delta D_F(x^k,x^\star)$.
\end{assumption}
Assumption~\ref{as:sto-grad} is satisfied by several stochastic gradient estimators used in machine learning, including some types of coordinate descent~\cite{wri15}, variance reduction~\cite{gow20a}, 
and also compressed gradients used to reduce the communication cost in distributed optimization~\cite{bas20,sat20,xu21}, see Table 1 in~\cite{gor20}. Also, the full gradient estimator defined by
    $g^{k+1} = \nabla F(x^k)$
satisfies Assumption~\ref{as:sto-grad} with $\alpha = \nu$, the smoothness constant of $F$, $\sigma_k \equiv 0$, $\rho = 1$, and $\delta = \beta = 0$, see Theorem~2.1.5 in~\cite{nesterov2018lectures}. The loopless SVRG estimator~\cite{hofmann2015variance,kovalev2019don} also satisfies Assumption~\ref{as:sto-grad}:
 \begin{proposition}[Loopless SVRG estimator]\label{proplsvrg}
    Assume that $F$ is written as a sum $F = \frac{1}{n}\sum_{i = 1}^n f_i,$ for some $n\geq 1$, where for every $i \in \{1,\ldots,n\}$, $f_i : \cX \to \bR$ is a $\nu_i$-smooth convex function. Let $p \in (0,1)$, and $(\Omega,\cF,\bP)$ be a probability space. On $(\Omega,\cF,\bP)$, consider:
    
    \noindent$\bullet\ \ $a sequence of i.i.d. random variables $(\theta^k)_{k\in\mathbb{N}}$ with Bernoulli distribution of parameter $p$,
    
     \noindent$\bullet\ \ $a sequence of i.i.d random variables $(\zeta^k)_{k\in\mathbb{N}}$ with uniform distribution over $\{1,\ldots,n\}$,

    \noindent$\bullet\ \ $the sigma-field $\cF_k$ generated by $(\theta^k,\zeta^k)_{0 \leq j \leq k}$ and a $(\cF_k)$-adapted stochastic process $(x^k)_{k\in\mathbb{N}}$,

  \noindent$\bullet\ \ $a stochastic process $(\tilde{x}^k)_{k\in\mathbb{N}}$ defined by
        $\tilde{x}^{k+1} = \theta^{k+1} x^k + (1-\theta^{k+1})\tilde{x}^{k}$,
   
  \noindent$\bullet\ \ $a stochastic process $(g^{k})_{k\in\mathbb{N}}$ defined by
        $g^{k+1} = \nabla f_{\zeta^{k+1}}(x^k) - \nabla f_{\zeta^{k+1}}(\tilde{x}^{k}) + \nabla F(\tilde{x}^{k})$.
   
Then, the process $(g^k)_{k\in\mathbb{N}}$ satisfies Assumption~\ref{as:sto-grad} with $\alpha = 2\max_{i \in \{1,\ldots,n\}} \nu_i$, $\beta = 2$, $\rho = p$, $\delta = \alpha p /2$, and $\sigma_k^2 = \frac{1}{n}\sum_{i=1}^n \bE_k \|\nabla f_i(\tilde{x}^{k}) - \nabla f_i(x^\star)\|^2$.\end{proposition}
\begin{proof}
    The proof is the same as the proof of Lemma~A.11 of~\cite{gor20}, which is only stated for $(x^k)_{k\in\mathbb{N}}$  generated by a specific algorithm, but remains true for any $(\cF_k)$-adapted stochastic process $(x^k)_{k\in\mathbb{N}}$.
\end{proof}

We can now exhibit our main results.  
In a nutshell,  $P^{-1}C(z^k)$ is replaced by the random realization $P^{-1}(g^{k+1},0)$
    and the last term of Eqn.~\eqref{eq:funda}, which is nonnegative,      is handled using Assumption~\ref{as:sto-grad}.

\subsection{The Stochastic PD3O Algorithm}

The Stochastic PD3O Algorithm, shown above,  has $\cO(1/k)$ ergodic convergence in the general case. A linear convergence result in the strongly convex setting is derived in Appendix~\ref{secappb}.
\begin{theorem}[Convergence of the Stochastic PD3O Algorithm]\label{th:cvx:PD3O}
Suppose that Assumption~\ref{as:sto-grad} holds.
Let $\kappa \eqdef \beta/\rho$, $\gamma, \tau >0$ be such that $\gamma \leq 1/{2(\alpha+\kappa\delta)}$ and $\gamma\tau\|L\|^2 < 1$.
Set $V^0 \eqdef \|v^{0} - v^\star\|_P^2 + \gamma^2 \kappa \sigma_{0}^2,$ where $v^0 = (p^0,y^0)$.
Then,  for every $k\in\mathbb{N}$,
\begin{equation*}
     \bE\left(\cL(\bar{x}^{k},y^\star) - \cL(x^\star,\bar{y}^{k+1})\right) \leq \frac{V^0}{k \gamma},
\end{equation*}
where $\bar{x}^{k} = \frac{1}{k} \sum_{j = 0}^{k-1} x^j$ and $\bar{y}^{k+1} = \frac{1}{k} \sum_{j = 1}^{k} y^j$.
\end{theorem}
\begin{proof}
    Using Lemma~\ref{lem:ineq:PD3O}, the convexity of $F$, $R$, $H^*$, and Lemma~\ref{lem:duality-gap},
\begin{align*}
    \bE_k \|v^{k+1} - v^\star\|_P^2 &+ \kappa \gamma^2\bE_k\sigma_{k+1}^2 \leq \|v^k - v^\star\|_P^2 + \kappa \gamma^2\left(1-\rho+\frac{\beta}{\kappa }\right)\sigma_k^2\notag\\
    &\quad-2\gamma(1-\gamma(\alpha  +\kappa \delta))\bE_k\left(\cL(x^k,d^{\star}) - \cL(x^\star,d^{k+1})\right).
\end{align*}
We have $1 - \rho + \beta/\kappa  = 1$, $\gamma \leq 1/2(\alpha  +\kappa \delta)$. Set 
$
V^k \eqdef \|v^{k} - v^\star\|_P^2 + \kappa \gamma^2\sigma_{k}^2
$, for every $k\in\mathbb{N}$.  
Then
$
    \bE_k V^{k+1} \leq V^k -\gamma \bE_k\left(\cL(x^{k},d^\star) - \cL(x^\star,d^{k+1})\right)
$. 
Taking the expectation,
$
    \gamma\bE\left(\cL(x^{k},d^\star) - \cL(x^\star,d^{k+1})\right) \leq \bE V^k - \bE V^{k+1}
$. 
Iterating and using the nonnegativity of $V^k$,
$
    \gamma \sum_{j = 0}^{k-1} \bE\left(\cL(x^{j},d^\star) - \cL(x^\star,d^{j+1})\right) \leq \bE V^0
$. 
Finally, note that $d^{k+1} = y^{k+1}$ and $d^{\star} = y^{\star}$. Indeed, $y^k = q^k$ and $q^{k+1} = q^k + d^{k+1} - y^k = d^{k+1}$.
We can conclude using the convex-concavity of $\cL$.
\end{proof}
 In the deterministic case $g^{k+1} = \nabla F(x^k)$, we recover the same rate as in  \cite[Theorem 2]{yan18}.
\begin{remark}[Primal--Dual gap]
    Deriving a similar bound on the stronger primal--dual gap $(F+R+H\circ L)(\bar{x}^{k})+\big((F+R)^*\circ (-L)+H^*\big)(\bar{y}^{k})$ 
    requires additional assumptions; for instance, even for the Chambolle--Pock algorithm, which is the particular case of the PD3O, PPDY and Condat--V\~u algorithms when $F=0$, the best available result \cite[Theorem 1]{cha162} is not stronger than Theorem \ref{th:cvx:PD3O} 
 \end{remark}
\begin{remark}[Particular case of SGD]
    In the case where $H =0$ and $L = 0$, the Stochastic PD3O Algorithm boils down to proximal stochastic gradient descent (SGD) and Theorem~\ref{th:cvx:PD3O} implies that
        $\bE\left((F+R)(\bar{x}^{k}) - (F+R)(x^\star)\right) \leq V^0/(\gamma k)
        $. 
   This $\cO(1/k)$ ergodic convergence rate 
   unifies known results on SGD in the non-strongly-convex case, whenever the stochastic gradient satisfies Assumption~\ref{as:sto-grad}. 
   \end{remark}

\subsection{The Stochastic PDDY Algorithm}
We now analyze the proposed Stochastic PDDY Algorithm, shown above. For it too, we have $\cO(1/k)$ ergodic convergence in the general case. A linear convergence result in the strongly convex setting is derived in Appendix~\ref{secappb}.
    \begin{theorem}[Convergence of the Stochastic PDDY Algorithm]
        \label{th:cvx:PDDY}
        Suppose that Assumption~\ref{as:sto-grad} holds.
        Let $\kappa \eqdef \beta/\rho$, $\gamma, \tau >0$ be such that $\gamma \leq 1/{2(\alpha+\kappa\delta)}$ and $\gamma\tau\|L\|^2 < 1$.
        Define $V^0 \eqdef \|v^{0} - v^\star\|_P^2 + \gamma^2 \kappa \sigma_{0}^2,$ where $v^0 = (p^0,y^0)$.
        Then, for every $k\in\mathbb{N}$,
        \begin{equation*}
             \bE\left(D_F(\bar{x}^{k},x^\star)+D_{H^*}(\bar{y}^{k+1},y^\star)+ D_R(\bar{s}^{k+1},s^\star)\right) \leq \frac{V^0}{k \gamma},
        \end{equation*}
        where $\bar{x}^{k} = \frac{1}{k} \sum_{j = 0}^{k-1} x^j$, $\bar{y}^{k+1} = \frac{1}{k} \sum_{j = 1}^{k} y^j$ and $\bar{s}^{k+1} = \frac{1}{k} \sum_{j = 1}^{k} s^j$.
\end{theorem}
\begin{proof}
Using Lemma~\ref{lem:PDDY} and the convexity of $F$, $R$, $H^*$,
\begin{align*}
    \bE_k &\|v^{k+1} - v^\star\|_P^2 + \kappa \gamma^2\bE_k\sigma_{k+1}^2 \leq \|v^k - v^\star\|_P^2 + \kappa \gamma^2\left(1-\rho+\frac{\beta}{\kappa }\right)\sigma_k^2\\
    &-2\gamma\big(1-\gamma(\alpha  +\kappa \delta)\big) \left(D_F(x^{k},x^\star) + D_{H^*}(y^{k},y^\star)+\bE_k D_R(s^{k+1},s^\star)\right).
\end{align*}
Since $1 - \rho + \beta/\kappa  = 1$, $\gamma \leq 1/2(\alpha  +\kappa \delta)$. Set 
$
V^k \eqdef \|v^{k} - v^\star\|_P^2 + \kappa \gamma^2\sigma_{k}^2
$. 
Then
\begin{equation*}
    \bE_k V^{k+1} \leq V^k -\gamma \bE_k\left(D_F(x^{k},x^\star) + D_{H^*}(y^{k},y^\star)+D_R(s^{k+1},s^\star)\right).
\end{equation*}
Taking the expectation, 
$
    \gamma\bE \left(D_F(x^{k},x^\star) + D_{H^*}(y^{k},y^\star)+D_R(s^{k+1},s^\star)\right) \leq \bE V^k - \bE V^{k+1}
$. 
Iterating and using the nonnegativity of $V^k$,
$
    \gamma \sum_{j = 0}^{k-1} \bE \big(D_F(x^{k},x^\star)+{}$ $ D_{H^*}(y^{k},y^\star)+D_R(s^{k+1},s^\star)\big) \leq \bE V^0
$. 
We conclude using the convexity of the Bregman divergence in its first variable.
\end{proof}

\section{Linearly Constrained Smooth Optimization}\label{sec7}

In this section, we consider the problem
\begin{equation}
\minimize_{x\in\mathcal{X}} \,F(x)\quad \mbox{s.t.}\quad Lx=b,\label{pbconstr0}
\end{equation}
where $L:\mathcal{X}\rightarrow \mathcal{Y}$ is a linear operator,
 $\mathcal{X}$ and  $\mathcal{Y}$ are real Hilbert spaces, $F$ is a $\nu$-smooth convex function, for some $\nu>0$, 
 and  $b \in \ran(L)$, the range of $L$. This is a particular case of Problem~\eqref{eq:original-pb}  with $R=0$ and 
 $H =\iota_b$. We suppose that a solution $x^\star$ exists, satisfying $Lx^\star=b$ and $0\in\nabla F(x^\star)+L^*y^\star$ for some $y^\star\in\cY$. 
 The stochastic PD3O and PDDY algorithms both revert to the same algorithm, shown above, which   
    we call  Linearly Constrained Stochastic Gradient Descent (LiCoSGD). It is fully split: it does not make use of projections  onto the affine space $\{x \in \cX, L x = b\}$ and only makes calls to $L$ and $L^*$. 
    In the deterministic case $g^{k+1} = \nabla F(x^k)$, LiCoSGD reverts to an instance of the algorithm  first proposed by Loris and Verhoeven in~\cite{lor11} and rediscovered independently as the PDFP2O algorithm~\cite{che13} and the Proximal Alternating Predictor--Corrector (PAPC) algorithm~\cite{dro15}.  Convergence of this algorithm follows from Theorem \ref{th:pddy-cv}, see other results in \cite{con19,con20}. Thus, LiCoSGD is a stochastic extension of this algorithm, 
for which 
Theorem \ref{th:cvx:PD3O} becomes:
\begin{theorem}[Convergence of LiCoSGD]
\label{th:cvx:lico}
Suppose that Assumption~\ref{as:sto-grad} holds.
Let $\kappa \eqdef \beta/\rho$, $\gamma, \tau >0$ be such that $\gamma \leq 1/{2(\alpha+\kappa\delta)}$ and $\gamma\tau\|L\|^2 < 1$.
Set $V^0 \eqdef \|v^{0} - v^\star\|_P^2 + \gamma^2 \kappa \sigma_{0}^2,$ where $v^0 = (w^0,y^0)$.
Then, for every $k\in\mathbb{N}$,
\begin{equation}
     \bE\left(F(\bar{x}^{k})-F(x^\star)+\langle L\bar{x}^{k}-b,y^\star\rangle
     \right) \leq \frac{V^0}{k \gamma},
\end{equation}
where $\bar{x}^{k} = \frac{1}{k} \sum_{j = 0}^{k-1} x^j$, 
$x^\star$ and $y^\star$ are some primal and dual solutions.
\end{theorem}
The convex function $x\mapsto F(x)-F(x^\star)+\langle Lx-b,y^\star\rangle$ is nonnegative and its minimum is zero, attained at $x^\star$. Under additional assumptions, like strict convexity around $x^\star$, this function takes value zero only if $F(x)=F(x^\star)$ and $Lx=b$, so that $x$ is a solution.

We now state an important result: strong convexity of $F$ is sufficient to get linear convergence. We denote by $\omega(W)$ the smallest positive eigenvalue of a positive self-adjoint linear operator $W$. Then it is easy to show that for every $y \in \ran(L)$, $\omega(LL^*)\|y\|^2 \leq \|L^* y\|^2$. Also, $\omega(LL^*)=\omega(L^*L)$.

   \begin{theorem}[Linear convergence of LiCoSGD with $F$ strongly convex]
   \label{th:LV0}
   Suppose that Assumption~\ref{as:sto-grad} holds, that $F$ is $\mu_F$-strongly convex, for some $\mu_F> 0$, and  that $y^0 \in \ran(L)$. Let $x^\star$ be the unique solution of \eqref{pbconstr0}, $y^\star$ be the unique element of $\ran(L)$ such that $\nabla F(x^\star) + L^* y^{\star} = 0$. Suppose that $\gamma>0$ and $\tau >0$ are such that $\gamma\tau\|L\|^2 < 1$ and $\gamma \leq 1/{\alpha + \kappa\delta}$, for some $\kappa > \beta/\rho$.
Define, for every $k\in\mathbb{N}$,
   \begin{equation}
       \label{eq:lyapunov-lin}
       V^k \eqdef \|x^{k} - x^\star\|^2+ \left(1+\tau\gamma\omega(L^* L)\right)\|y^{k} - y^\star\|_{\gamma,\tau}^2 + \kappa\gamma^2\bE \sigma_{k}^2,
   \end{equation}
   and
   \begin{equation}
       \label{eq:rate-lin}
       r \eqdef \max\left(1-\gamma\mu_F,1-\rho+\frac{\beta}{\kappa},\frac{1}{1+\tau\gamma\omega(L^* L)}\right)<1.
   \end{equation}
  Then, for every $k\in\mathbb{N}$, $\bE V^{k} \leq r^k V^0$. 
   \end{theorem}

\begin{proof}
Noting that $y^\star = d^\star = q^\star$ and applying Lemma~\ref{lem:ineq:PD3O} with $\gamma \leq (\alpha  +\kappa \delta)$,
\begin{align*}
    &\bE_k \|p^{k+1} - p^\star\|^2 + \bE_k\|q^{k+1} - q^\star\|_{\gamma,\tau}^2 +  \kappa \gamma^2\bE_k\sigma_{k+1}^2 \leq \|p^{k} - p^\star\|^2 + \|q^{k} - q^\star\|_{\gamma,\tau}^2 \\
    &\quad-\gamma\mu_F\|x^{k}-x^\star\|^2+ \kappa \gamma^2\left(1-\rho+\frac{\beta}{\kappa }\right)\sigma_k^2-\gamma^2\|P^{-1}A(u^{k+1}) - P^{-1}A(u^{\star})\|_P^2.\notag
\end{align*}
Since the component of $P^{-1}A(u^{k+1}) - P^{-1}A(u^{\star})$ in $\cX$ is $L^* d^{k+1} - L^* d^{\star}$, we have
\begin{align*}
    \bE_k \|p^{k+1} &- p^\star\|^2 + \bE_k \|q^{k+1} - q^\star\|_{\gamma,\tau}^2 +  \kappa \gamma^2\bE_k\sigma_{k+1}^2 \leq \|x^{k} - x^\star\|^2 + \|q^{k} - q^\star\|_{\gamma,\tau}^2 \\
    &\quad-\gamma\mu_F\|p^{k}-p^\star\|^2+ \kappa \gamma^2\left(1-\rho+\frac{\beta}{\kappa }\right)\sigma_k^2-\gamma^2\|L^* d^{k+1} - L^* d^{\star}\|^2.\notag
\end{align*}
Inspecting the iterations of the algorithm, one can see that $d^{0} \in \ran(L)$ implies $d^{k+1} \in \ran(L)$. Since $d^\star \in \ran(L)$, $d^{k+1} - d^\star \in \ran(L)$. Therefore, 
$\omega(LL^*)\|d^{k+1} - d^{\star}\|^2 \leq \|L^* d^{k+1} - L^* d^{\star}\|^2$. Since $q^{k+1} = d^{k+1} = y^{k+1}$ and $x^k = p^k$,
\begin{align*}
    &\bE_k \|x^{k+1} - x^\star\|^2 + (1+\gamma\tau\omega(LL^*))\bE_k\|y^{k+1} - y^\star\|_{\gamma,\tau}^2 + \kappa \gamma^2\bE_k\sigma_{k+1}^2\notag\\
   &\quad\leq (1-\gamma\mu_F)\|x^{k} - x^\star\|^2 + \|y^{k} - y^\star\|_{\gamma,\tau}^2+ \kappa \gamma^2\left(1-\rho+\frac{\beta}{\kappa}\right)\sigma_k^2.
    \end{align*}
Thus, by setting $V^k$ as in  \eqref{eq:lyapunov-lin} and $r$ as in \eqref{eq:rate-lin}, we have $
    \bE_k V^{k+1} \leq r V^k$.
\end{proof}

   To the best of our knowledge, even in the deterministic case (with $\alpha = \nu$, $\rho = 1$, $\delta = \beta = 0$, $\kappa=1$), this is a first time that a fully split algorithm using $\nabla F$, $L$ and $L^*$ is shown to converge linearly to a solution of \eqref{pbconstr0}, whenever $F$ is strongly convex. Also, the knowledge of $\mu_F$ is not needed.
   We discuss the application of LiCoSGD to decentralized optimization in Appendix~\ref{apppridec}.

\section{Experiments}\label{sec:exp}

We present numerical experiments for the PDDY, PD3O and Condat--V\~u (CV)~\cite[Algorithm 3.1]{con13} algorithms. We observed that the performance of these algorithms is nearly identical, when the same stepsizes are used; but the PDDY and PD3O algorithms have a larger range of stepsizes than the CV algorithm, so that they are often faster after tuning. 
We used $\gamma\tau \|L\|^2=0.999$, which was always the best choice for these two algorithms. 
So, we do not provide direct comparisons in the plots. Instead, we focus on how the choice of the stochastic gradient estimator affects the convergence speed; we compare the true gradient, the standard stochastic gradient estimator (SGD), the VR estimators SAGA \cite{def14} and SVRG \cite{joh13a,zha13,xia14}. 
We used closed-form expressions for $\nu$ and tuned the stepsizes for all methods by running logarithmic grid search with factor 1.5 over multiples of $\frac{1}{\nu}$. We used a batch size of 16 for better parallelism in the stochastic estimators.
For SGD, we used a small value of $\gamma$, such as $\frac{0.01}{\nu}$. 
\smallskip

\begin{figure}[t]
	\centering
	\includegraphics[scale=0.36]{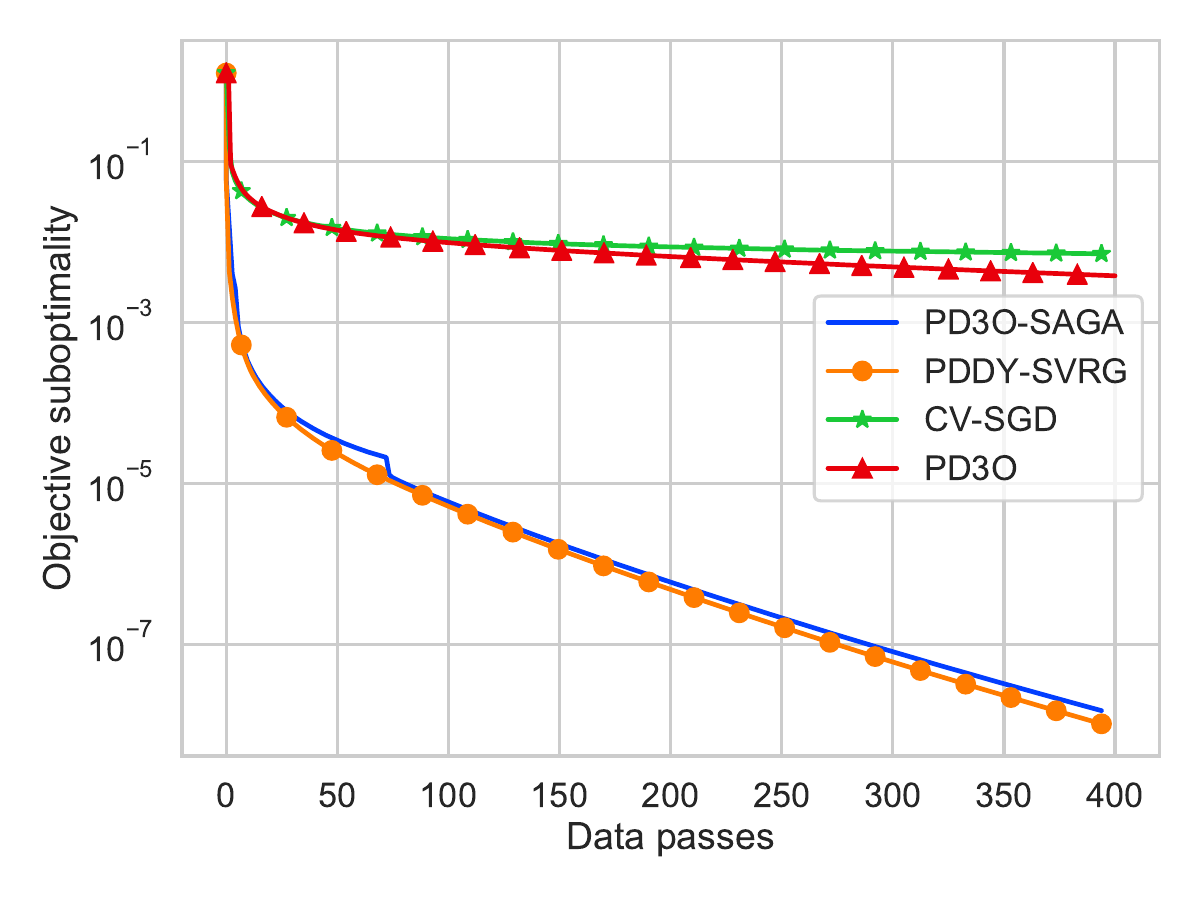}
	\includegraphics[scale=0.36]{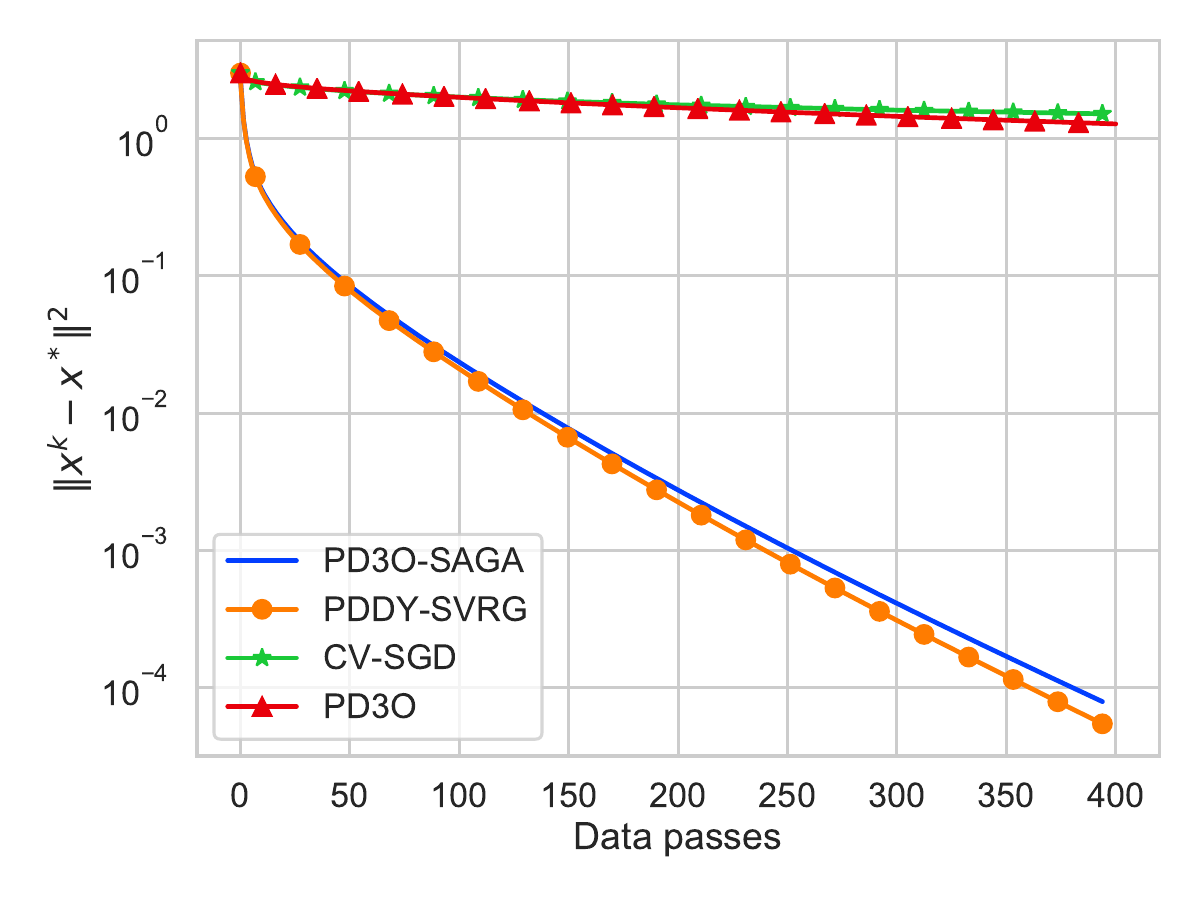}
	\caption{Results for the PCA-Lasso experiment. Left: convergence w.r.t.\ the objective function; right: convergence in norm.
}\label{fig1}
\end{figure}

\subsection{PCA-Lasso}

In a recent work~\cite{tay21}, the difficult PCA-based Lasso problem was considered: 
$\min_x \frac{1}{2}\|Wx - a\|^2 + \lambda\|x\|_1 + \lambda_1\sum_{i=1}^m \|L_i x\|$, 
where $W\in\mathbb{R}^{n\times p}$, $a\in\mathbb{R}^n$, $\lambda, \lambda_1>0$ are given.
 We generated 10 matrices $L_i$  randomly with standard normal i.i.d.\ entries, each with 20 rows. $W$ and $y$ were taken from the `mushrooms' dataset in the libSVM base~\cite{chang2011libsvm}. We chose $\lambda=\frac{\nu}{10n}$ and $\lambda_1=\frac{2\nu}{nm}$, where $\nu$ is needed to compensate for the fact that we do not normalize the objective. The results are shown in Fig.~\ref{fig1}. The advantage of using a VR stochastic gradient estimate is clear, with SAGA and SVRG being very similar.

\begin{figure}[t]
	\centering
	\includegraphics[scale=0.36]{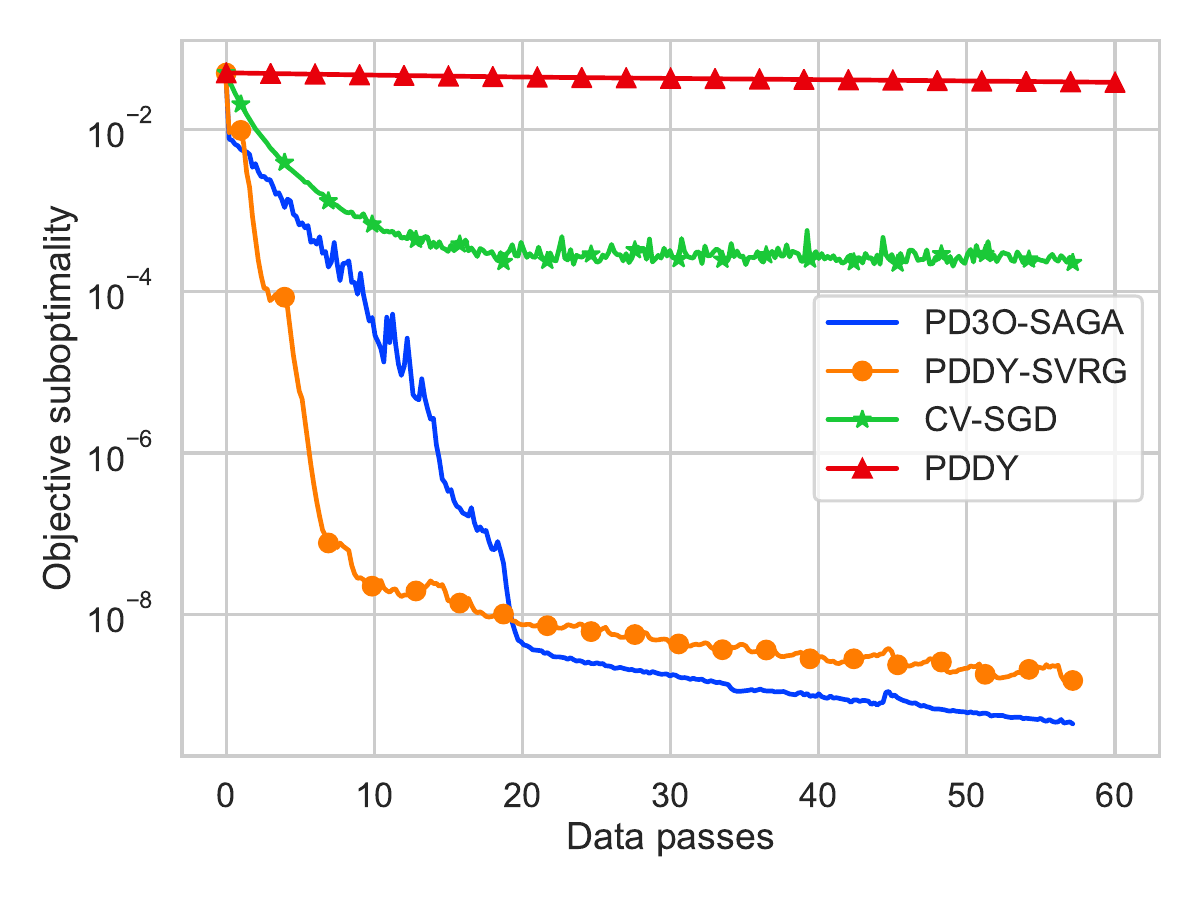}
	\includegraphics[scale=0.36]{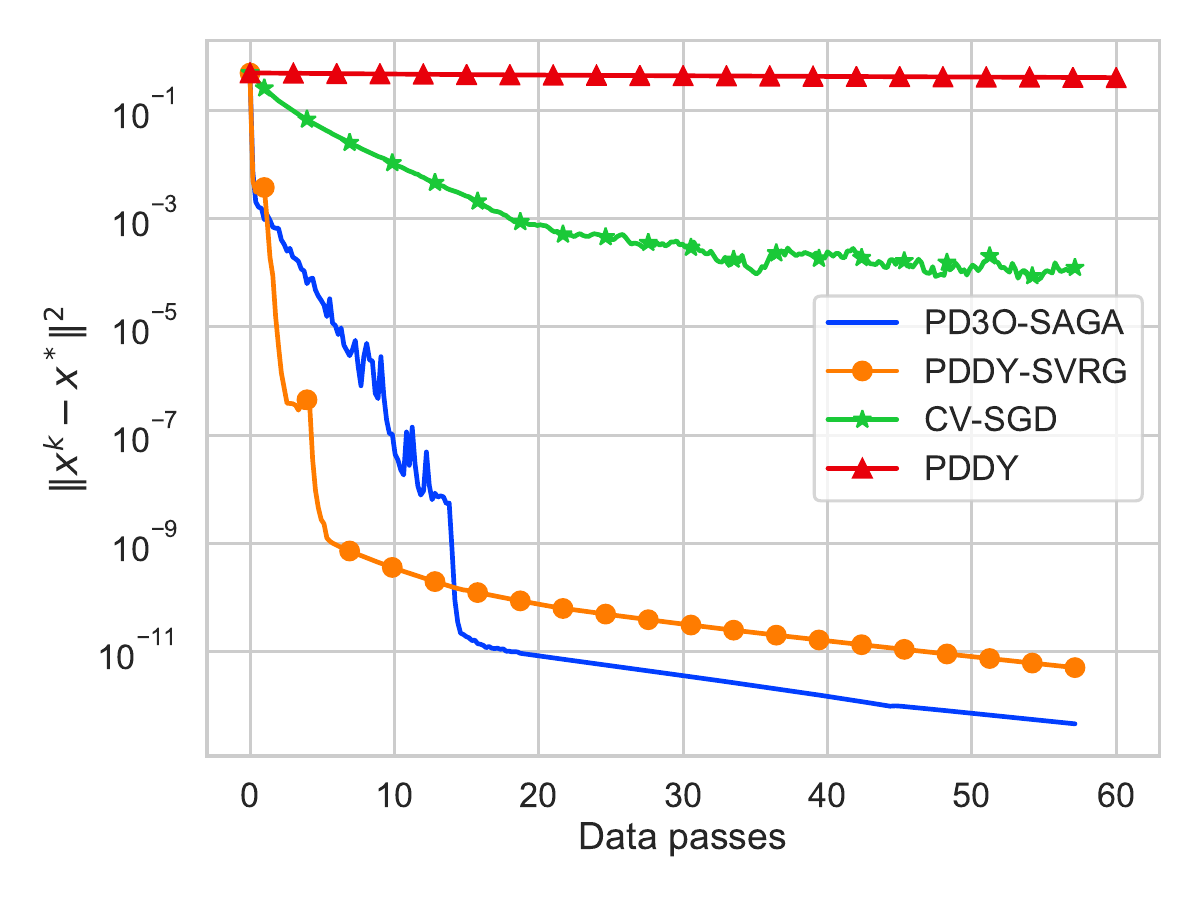}
	\caption{Results for the MNIST experiment. Left: convergence  w.r.t.\ the objective function; right: convergence in norm.}
	\label{fig2}
\end{figure}

\subsection{MNIST with Overlapping Group Lasso}

We consider the problem where $F$ is the $\ell_2$-regularized logistic loss and a group Lasso penalty. Given the data matrix $W\in \mathbb{R}^{n\times p}$ and vector of labels $a\in\{0,1\}^n$, $F(x)=\frac{1}{n}\sum_{i=1}^n f_i(x) + \frac{\lambda}{2}\|x\|^2$, 
where	$f_i(x)=-\big(a_i \log \big(h(w_i^\top x)\big) + (1-a_i)\log\big(1-h(w_i^\top x)\big)\big)$,
$\lambda=\frac{2\nu}{n}$, $w_i\in\mathbb{R}^p$ is the $i$-th row of $W$, and $h:t\to1/(1+e^{-t})$. 
The nonsmooth regularizer 
is given by
	$\lambda_1\sum_{j=1}^m \|x\|_{G_j}$, 
where $\lambda_1=\frac{\nu}{5n}$, $G_j\subset \{1,\dotsc, p\}$ is a given subset of coordinates and $\|x\|_{G_j}$ is the $\ell_2$-norm of the corresponding block of $x$. To apply splitting methods, we use $L=(I_{G_1}^\top, \dotsc, I_{G_m}^\top)^\top$, where $I_{G_j}$ is the operator that takes $x\in\mathbb{R}^p$ and returns only the entries from block $G_j$. Then, we can use $H(y)=\lambda_1\sum_{j=1}^m\|y\|_{G_j}$, which is separable in $y$ and, thus, proximable.
We use the MNIST dataset~\cite{lecun2010mnist} of 70000 black and white $28\times 28$ images. For each pixel, we add a group of pixels $G_j$ adjacent to it, including the pixel itself. Since there are some border pixels, groups consist of 3, 4 or 5 coordinates, and there are 784 penalty terms in total. The results are shown in Fig.~\ref{fig2}. Here SAGA is a bit better than SVRG.

\subsection{Fused Lasso}

In the Fused Lasso problem, we are given a feature matrix $W\in\mathbb{R}^{n\times p}$ and an output vector $a$, which define the least-squares function $F(x)=\frac{1}{2}\|Wx-a\|^2$. It is regularized with $\frac{\lambda}{2}\|x\|^2$ and $\lambda_1\|Dx\|_1$, where $\lambda=\frac{\nu}{n}$, $\lambda_1=\frac{\nu}{10n}$ and $D\in\mathbb{R}^{(p-1)\times p}$ has entries $D_{i,i}=1$, $D_{i, i+1}=-1$, for $i=1,\dotsc, p-1$, and $D_{ij}=0$ otherwise. We used again the `mushrooms' dataset. The plots look very similar to the ones in Fig.~\ref{fig1}, so we omit them.

\section{Conclusion}

We proposed a new primal--dual proximal splitting algorithm, the Primal--Dual Davis--Yin (PDDY) algorithm, to minimize a sum of three functions, one of which is composed with a linear operator. It is an alternative to the PD3O algorithm; they often perform similarly, but one or the other may be preferable for the problem at hand, depending on the implementation details. In particular, their memory requirements can be different. Furthermore, we proposed stochastic variants of both algorithms, studied their convergence rates, and showed by experiments that they can be much faster than their deterministic counterparts. We also showed that for linearly-constrained minimization of a strongly convex function, an instance of the stochastic PDDY algorithm, called LiCoSGD, converges linearly. 
We studied all algorithms within the unified framework of a stochastic generalization of Davis--Yin splitting for monotone inclusions. Our machinery opens the door to a promising class of new randomized proximal algorithms for large-scale optimization.

\appendix
\section*{Appendix}

\section{Lemmas A.1 and A.2}\label{secappa}

We state Lemma~\ref{lem:ineq:PD3O} and Lemma~\ref{lem:PDDY}, which are used in the proofs of Theorem~\ref{th:cvx:PD3O} and Theorem~\ref{th:cvx:PDDY}, respectively.

{To simplify the notations, we use the following convention: when a set appears in an equation while a single element is expected, e.g.\ $\partial R (x^k)$, this means that the equation holds for some element in this nonempty set.}

\begin{lemma}
\label{lem:ineq:PD3O}
Assume that $F$ is $\mu_F$-strongly convex, for some $\mu_F \geq 0$, and that $(g^k)_{k\in\mathbb{N}}$ satisfies Assumption~\ref{as:sto-grad}. Then the iterates of the Stochastic PD3O Algorithm satisfy
\begin{align}
    \bE_k \|v^{k+1} - v^\star\|_P^2 &+ \kappa \gamma^2\bE_k\sigma_{k+1}^2 \leq \|v^k - v^\star\|_P^2 + \kappa \gamma^2\left(1-\rho+\frac{\beta}{\kappa }\right)\sigma_k^2\notag \\
    &-2\gamma(1-\gamma(\alpha  +\kappa \delta)) D_F(x^{k},x^\star)-\gamma\mu_F\|x^{k}-x^\star\|^2\\
    &-2\gamma\ps{\partial R(x^{k}) - \partial R(x^\star),x^{k} - x^\star}-2\gamma\bE_k\ps{\partial H^*(d^{k+1}) - \partial H^*(d^\star),d^{k+1} - d^\star}\notag\\
    &-\gamma^2\bE_k\big\|P^{-1}A(u^{k+1})+P^{-1}B(z^{k})- \left(P^{-1}A(u^{\star})+P^{-1}B(z^{\star})\right)\big\|_P^2.\notag
\end{align}
\end{lemma}

\begin{proof}
Applying Lemma~\ref{lem:funda-DYS} for $\text{DYS}(P^{-1}A,P^{-1}B,P^{-1}C)$ using the norm induced by $P$, we have
\begin{align*}
    \|v^{k+1} - v^\star\|_P^2 ={}& \|v^k - v^\star\|_P^2 -2\gamma\ps{P^{-1}B(z^{k}) - P^{-1}B(z^\star),z^{k} - z^\star}_P\\
    &-2\gamma\ps{P^{-1}C(z^{k}) - P^{-1}C(z^\star),z^{k} - z^\star}_P+\gamma^2\|P^{-1}C(z^{k}) - P^{-1}C(z^\star)\|_P^2\\
    &-2\gamma\ps{P^{-1}A(u^{k+1}) - P^{-1}A(u^\star),u^{k+1} - u^\star}_P\\
    &-\gamma^2\|P^{-1}A(u^{k+1})+P^{-1}B(z^{k}) - \left(P^{-1}A(u^{\star})+P^{-1}B(z^{\star})\right)\|_P^2\\
    {}={}& \|v^k - v^\star\|_P^2 -2\gamma\ps{B(z^{k}) - B(z^\star),z^{k} - z^\star}+\gamma^2\|P^{-1}C(z^{k}) - P^{-1}C(z^\star)\|_P^2\\
    &-2\gamma\ps{C(z^{k}) - C(z^\star),z^{k} - z^\star}-2\gamma\ps{A(u^{k+1}) - A(u^\star),u^{k+1} - u^\star}\\
    &-\gamma^2\|P^{-1}A(u^{k+1})+P^{-1}B(z^{k}) - \left(P^{-1}A(u^{\star})+P^{-1}B(z^{\star})\right)\|_P^2.
\end{align*}
Using $
    A(u^{k+1}) =  \big(  L^* d^{k+1} , -L s^{k+1}+\partial H^*(d^{k+1})\big),
    B(z^k) =  \big(\partial R(x^{k}) ,0 \big),
    C(z^k) = \big( g^{k+1},0\big)
$
and
$
    A(u^{\star}) =  \big(  L^* d^\star , -L s^\star+\partial H^*(d^\star)\big),
    B(z^\star) =  \big(\partial R(x^\star) ,0 \big),
    C(z^\star) = \big(  \nabla F(x^\star),0\big)
$, 
we have
\begin{align*}
    \|v^{k+1} - v^\star\|_P^2 ={}& \|v^k - v^\star\|_P^2 -2\gamma\ps{\partial R(x^{k}) - \partial R(x^\star),x^{k} - x^\star}+\gamma^2\|g^{k+1} - \nabla F(x^\star)\|^2\\
    &-2\gamma\ps{g^{k+1} - \nabla F(x^\star),x^{k} - x^\star}-2\gamma\ps{\partial H^*(d^{k+1}) - \partial H^*(d^\star),d^{k+1} - d^\star}\\
    &-\gamma^2\|P^{-1}A(u^{k+1})+P^{-1}B(z^{k}) - \left(P^{-1}A(u^{\star})+P^{-1}B(z^{\star})\right)\|_P^2.\notag
\end{align*}
Taking conditional expectation w.r.t.\ $\cF_k$ and using Assumption~\ref{as:sto-grad},
\begin{align*}
    \bE_k \|v^{k+1} - v^\star\|_P^2 \leq{}& \|v^k - v^\star\|_P^2 -2\gamma\ps{\partial R(x^{k}) - \partial R(x^\star),x^{k} - x^\star}\notag\\
    &-2\gamma\ps{\nabla F(x^{k}) - \nabla F(x^\star),x^{k} - x^\star}\notag\\
    &-2\gamma\bE_k\ps{\partial H^*(d^{k+1}) - \partial H^*(d^\star),d^{k+1} - d^\star}+\gamma^2 \left(2\alpha   D_F(x^{k},x^\star) + \beta\sigma_k^2\right)\notag\\
    &-\gamma^2\bE_k\|P^{-1}A(u^{k+1})+P^{-1}B(z^{k}) - \left(P^{-1}A(u^{\star})+P^{-1}B(z^{\star})\right)\|_P^2.\notag
\end{align*}
Using strong convexity of $F$,
\begin{align*}
    \bE_k \|v^{k+1} - v^\star\|_P^2 \leq{}& \|v^k - v^\star\|_P^2 -\gamma\mu_F\|x^{k}-x^\star\|^2-2\gamma D_F(x^{k},x^\star)\notag\\
    &+\gamma^2 \left(2\alpha  D_F(x^{k},x^\star) + \beta\sigma_k^2\right)-2\gamma\ps{\partial R(x^{k}) - \partial R(x^\star),x^{k} - x^\star}\notag\\
    &-2\gamma\bE_k\ps{\partial H^*(d^{k+1}) - \partial H^*(d^\star),d^{k+1} - d^\star}\notag\\
    &-\gamma^2\bE_k\|P^{-1}A(u^{k+1})+P^{-1}B(z^{k}) - \left(P^{-1}A(u^{\star})+P^{-1}B(z^{\star})\right)\|_P^2.\notag
\end{align*}
Using Assumption~\ref{as:sto-grad},
\begin{align*}
    \bE_k \|v^{k+1} - v^\star\|_P^2 &+ \kappa \gamma^2\bE_k\sigma_{k+1}^2 \leq \|v^k - v^\star\|_P^2 + \kappa \gamma^2\left(1-\rho+\frac{\beta}{\kappa }\right)\sigma_k^2-\gamma\mu_F\|x^{k}-x^\star\|^2\notag\\
    &-2\gamma(1-\gamma(\alpha  +\kappa \delta)) D_F(x^{k},x^\star)-2\gamma\ps{\partial R(x^{k}) - \partial R(x^\star),x^{k} - x^\star}\\
    &-2\gamma\bE_k\ps{\partial H^*(d^{k+1}) - \partial H^*(d^\star),d^{k+1} - d^\star}\notag\\
    &-\gamma^2\bE_k\big\|P^{-1}A(u^{k+1})+P^{-1}B(z^{k}) - \left(P^{-1}A(u^{\star})+P^{-1}B(z^{\star})\right)\big\|_P^2.
\end{align*}
\end{proof}
\begin{lemma}
\label{lem:PDDY}
Suppose that $(g^k)_{k\in\mathbb{N}}$ satisfies Assumption~\ref{as:sto-grad}. Then the iterates of the Stochastic PDDY Algorithm satisfy
\begin{align*}
    \bE_k \|v^{k+1} - v^\star\|_P^2 &+ \kappa \gamma^2\bE_k\sigma_{k+1}^2 \leq \|v^k - v^\star\|_P^2 + \kappa \gamma^2\left(1-\rho+\frac{\beta}{\kappa }\right)\sigma_k^2\notag\\
    &-2\gamma(1-\gamma(\alpha  +\kappa \delta)) D_F(x^{k},x^\star)-2\gamma\ps{\partial H^*(y^{k}) - \partial H^*(y^\star),y^{k} - y^\star}\\
    &-2\gamma\bE_k\ps{\partial R(s^{k+1}) - \partial R(s^\star),s^{k+1} - s^\star}.\notag
\end{align*}
\end{lemma}
\begin{proof}
Applying Lemma~\ref{lem:funda-DYS} for $\text{DYS}(P^{-1}B,P^{-1}A,P^{-1}C)$ using the norm induced by $P$, we have
\begin{align*}
    \|v^{k+1} - v^\star\|_P^2 ={}& \|v^k - v^\star\|_P^2 -2\gamma\ps{P^{-1}A(z^{k}) - P^{-1}A(z^\star),z^{k} - z^\star}_P\\
    &-2\gamma\ps{P^{-1}C(z^{k}) - P^{-1}C(z^\star),z^{k} - z^\star}_P+\gamma^2\|P^{-1}C(z^{k}) - P^{-1}C(z^\star)\|_P^2\\
    &-2\gamma\ps{P^{-1}B(u^{k+1}) - P^{-1}B(u^\star),u^{k+1} - u^\star}_P\\
    &-\gamma^2\|P^{-1}B(u^{k+1})+P^{-1}A(z^{k}) - \left(P^{-1}B(u^{\star})+P^{-1}A(z^{\star})\right)\|_P^2\\
    ={}& \|v^k - v^\star\|_P^2-2\gamma\ps{A(z^{k}) - A(z^\star),z^{k} - z^\star}-2\gamma\ps{C(z^{k}) - C(z^\star),z^{k} - z^\star}\\
    &-2\gamma\ps{B(u^{k+1}) - B(u^\star),u^{k+1} - u^\star}+\gamma^2\|P^{-1}C(z^{k}) - P^{-1}C(z^\star)\|_P^2\\
    &-\gamma^2\|P^{-1}B(u^{k+1})+P^{-1}A(z^{k}) - \left(P^{-1}B(u^{\star})+P^{-1}A(z^{\star})\right)\|_P^2.
\end{align*}
Using $
    A(z^{k}) =  \big(  L^* y^k ,-L x^k+\partial H^*(y^k)\big),
    B(u^{k+1}) =  \big(\partial R(s^{k+1}) ,0 \big),
    C(z^k) = \big( g^{k+1},0\big)
$
and
$
    A(z^{\star}) =  \big(  L^* y^\star , -L x^\star+\partial H^*(y^\star)\big),
    B(u^\star) =  \big(\partial R(s^\star) ,0 \big),
    C(z^\star) = \big(  \nabla F(x^\star),0\big)
$, 
we have,
\begin{align*}
    \|v^{k+1} - v^\star\|_P^2 \leq{}& \|v^k - v^\star\|_P^2-2\gamma\ps{\partial H^*(y^{k}) - \partial H^*(y^\star),y^{k} - y^\star}+\gamma^2\|g^{k+1} - \nabla F(x^\star)\|^2\\
    &-2\gamma\ps{g^{k+1} - \nabla F(x^\star),x^{k} - x^\star}-2\gamma\ps{\partial R(s^{k+1}) - \partial R(s^\star),s^{k+1} - s^\star}.
\end{align*}
Applying the conditional expectation w.r.t. $\cF_k$ and using Assumption~\ref{as:sto-grad},
\begin{align*}
    \bE_k \|v^{k+1} - v^\star\|_P^2 \leq{}& \|v^k - v^\star\|_P^2-2\gamma\ps{\partial H^*(y^{k}) - \partial H^*(y^\star),y^{k} - y^\star}\notag\\
    &-2\gamma\ps{\nabla F(x^{k}) - \nabla F(x^\star),x^{k} - x^\star}+\gamma^2 \left(2\alpha   D_F(x^{k},x^\star) + \beta\sigma_k^2\right)\\
    &-2\gamma\bE_k\ps{\partial R(s^{k+1}) - \partial R(s^\star),s^{k+1} - s^\star}.\notag\end{align*}
Using the convexity of $F$,
\begin{align*}
    \bE_k \|v^{k+1} - v^\star\|_P^2 \leq{}& \|v^k - v^\star\|_P^2-2\gamma\ps{\partial H^*(y^{k}) - \partial H^*(y^\star),y^{k} - y^\star}-2\gamma D_F(x^{k},x^\star)\\
    &-2\gamma\bE_k\ps{\partial R(s^{k+1}) - \partial R(s^\star),s^{k+1} - s^\star}+\gamma^2 \left(2\alpha   D_F(x^{k},x^\star) + \beta\sigma_k^2\right).\notag
\end{align*}
Using Assumption~\ref{as:sto-grad},
\begin{align*}
    \bE_k \|v^{k+1} - v^\star\|_P^2 &+ \kappa \gamma^2\bE_k\sigma_{k+1}^2 \leq \|v^k - v^\star\|_P^2 + \kappa \gamma^2\left(1-\rho+\frac{\beta}{\kappa }\right)\sigma_k^2\notag\\
    &-2\gamma(1-\gamma(\alpha  +\kappa \delta)) D_F(x^{k},x^\star)-2\gamma\ps{\partial H^*(y^{k}) - \partial H^*(y^\star),y^{k} - y^\star}\\
    &-2\gamma\bE_k\ps{\partial R(s^{k+1}) - \partial R(s^\star),s^{k+1} - s^\star}.
\end{align*}\end{proof}

\section{Linear Convergence Results}\label{secappb}
In this section, we provide linear convergence results for the stochastic PD3O and the stochastic PDDY algorithms, in addition to Theorem \ref{th:LV0}. 
For an operator splitting method like DYS$(\tilde{A},\tilde{B},\tilde{C})$ to converge linearly, it is necessary that $\tilde{A}+\tilde{B}+\tilde{C}$ is strongly monotone. But this is not sufficient, and in general, 
 to converge linearly, DYS$(\tilde{A},\tilde{B},\tilde{C})$ requires the stronger assumption that $\tilde{A}$ or $\tilde{B}$ or $\tilde{C}$ is strongly monotone,  
 and in addition that $\tilde{A}$ or $\tilde{B}$ is cocoercive~\cite{dav17}. 
The PDDY algorithm is equivalent to 
DYS$(P^{-1}B,P^{-1}A,P^{-1}C)$ and the PD3O algorithm is equivalent to DYS$(P^{-1}A,P^{-1}B,P^{-1}C)$, see Sect.~\ref{sec:pdalgos}. However, $P^{-1}A$, $P^{-1}B$ and $P^{-1}C$ are not strongly monotone. In spite of this, we will prove  linear convergence of the (stochastic) PDDY and PD3O algorithms. 

Thus, for  both algorithms, we will make the assumption that $P^{-1}A+P^{-1}B+P^{-1}C$ is strongly monotone. This is equivalent to assuming that $M = A+B+C$ is strongly monotone; that is, that $F+R$ is strongly convex and $H$ is smooth. For instance, the Chambolle--Pock algorithm~\cite{cha11a,cha162}, which is a particular case of the PD3O and the PDDY algorithms, requires $R$ strongly convex and $H$ smooth to converge linearly, in general. 
In fact, for primal--dual algorithms to converge linearly on Problem \eqref{eq:original-pb}, for any $L$, it seems unavoidable that 
$F+R$ is strongly convex and that the dual term $H^*$ is strongly convex too, because the algorithm needs to be contractive in both the primal and the dual spaces. This means that $H$ must be smooth. We can remark  that if $H$ is smooth, it is tempting to use its gradient instead of its proximity operator. We can then use the proximal gradient algorithm to solve Problem \eqref{eq:original-pb} with $\nabla (F+H\circ L)(x)=\nabla F(x) + L^*\nabla H (Lx)$. However, in practice, it is often faster to use the proximity operator instead of the gradient, see a recent analysis of this topic in \cite{com19}.

For the PD3O algorithm, we will add a cocoercivity assumption, as suggested by the general linear convergence theory of DYS. More precisely, we will assume that $R$ is smooth, so that $P^{-1}B$ is cocoercive. Our result on the PD3O is therefore an extension of~\cite[Theorem 3]{yan18} to the stochastic setting. For the PDDY algorithm, this assumption is not needed to prove linear convergence, which is an advantage over the PD3O algorithm.

We denote by $\|\cdot\|_{\gamma,\tau}$ the norm induced by $\frac{\gamma}{\tau}I - \gamma^2 L L^*$ on $\cY$.

\begin{theorem}[Linear convergence of the Stochastic PD3O Algorithm]
    \label{th:Hsmooth:PD3O}
    Suppose that Assumption~\ref{as:sto-grad} holds. Suppose that $H$ is $1/\mu_{H^*}$-smooth, for some $\mu_{H^*} >0$, $F$ is $\mu_F$-strongly convex, for some $\mu_F\geq 0$,  and $R$ is $\mu_R$-strongly convex, for some $\mu_R\geq 0$, with $\mu \eqdef \mu_F + 2\mu_R >0$. Also, suppose that $R$ is $\lambda$-smooth, for some $\lambda>0$. Suppose that the parameters $\gamma>0$ and $\tau >0$ 
    satisfy $\gamma \leq 1/(\alpha+\kappa\delta)$, for some $\kappa > \beta/\rho$, 
   and  $\gamma\tau\|L\|^2 < 1$. Define, for every $k\in\mathbb{N}$,
    \begin{equation}
        \label{eq:lyapunov-smooth}
        V^k \eqdef \|p^{k} - p^\star\|^2+ \left(1+2\tau\mu_{H^*}\right)\|y^{k} - y^\star\|_{\gamma,\tau}^2 + \kappa\gamma^2 \sigma_{k}^2,
    \end{equation}
    and
    \begin{equation}
        \label{eq:rate-smooth}
        r \eqdef \max\left(1-\frac{\gamma\mu}{(1+\gamma \lambda)^2},\left(1-\rho+\frac{\beta}{\kappa}\right),\frac{1}{1+2\tau\mu_{H^*}}\right).
    \end{equation}
    Then, for every $k\in\mathbb{N}$,
   $
        \bE V^{k} \leq r^k V^0
    $.
\end{theorem}
\begin{proof}
We first use Lemma~\ref{lem:ineq:PD3O} along with the strong convexity of $R,H^*$. Note that $y^{k} = q^k$ and therefore $q^{k+1} = q^k + d^{k+1} - q^{k} = d^{k+1}$. We have
\begin{align*}
    &\bE_k \|p^{k+1} - p^\star\|^2 + \bE_k \|q^{k+1} - q^\star\|_{\gamma,\tau}^2 + 2\gamma\mu_{H^*}\bE_k\|q^{k+1} - q^\star\|^2 + \kappa \gamma^2\bE_k\sigma_{k+1}^2 \\
    &\quad\leq \|p^{k} - p^\star\|^2 + \|q^{k} - q^\star\|_{\gamma,\tau}^2 -\gamma\mu\|x^{k}-x^\star\|^2 + \kappa \gamma^2\left(1-\rho+\frac{\beta}{\kappa }\right)\sigma_k^2\\
    &\quad\quad-2\gamma(1-\gamma(\alpha  +\kappa \delta)) D_F(x^{k},x^\star).
\end{align*}
Noting that for every $q \in \mathcal{Y}$, $\|q\|_{\gamma,\tau}^2 = \frac{\gamma}{\tau}\|q\|^2 - \gamma^2\|L^* q\|^2 \leq \frac{\gamma}{\tau}\|q\|^2$, and taking $\gamma \leq 1/(\alpha  +\kappa \delta)$, we have
\begin{align}
    &\bE_k \|p^{k+1} - p^\star\|^2+ \left(1+2\tau\mu_{H^*}\right)\bE_k\|q^{k+1} - q^\star\|_{\gamma,\tau}^2 + \kappa \gamma^2\bE_k\sigma_{k+1}^2\notag\\
    &\quad\leq \|p^{k} - p^\star\|^2 + \|q^{k} - q^\star\|_{\gamma,\tau}^2 -\gamma\mu\|x^{k}-x^\star\|^2 + \kappa \gamma^2\left(1-\rho+\frac{\beta}{\kappa }\right)\sigma_k^2.\notag
\end{align}
Finally, since $R$ is $\lambda$-smooth, $\|p^k - p^\star\|^2 \leq (1+2 \gamma \lambda + \gamma^2 \lambda^2)\|x^{k}-x^\star\|^2$. Indeed, in this case, applying Lemma~\ref{lem:funda-DYS} with $\tilde{A} =0$, $\tilde{C} = 0$ and $\tilde{B} = \nabla R$, we obtain that if $x^k = \prox_{\gamma R}(p^k)$ and $x^\star = \prox_{\gamma R}(p^\star)$, then
\begin{align*}
\|x^k - x^\star\|^2 ={}& \|p^k - p^\star\|^2-2\gamma\ps{\nabla R(x^k) - \nabla R(x^\star),x^k - x^\star} - \gamma^2\|\nabla R(x^k) - \nabla R(x^\star)\|^2\\
\geq{}& \|p^k - p^\star\|^2-2\gamma\lambda \|x^k - x^\star\|^2 - \gamma^2\lambda^2\|x^{k} - x^\star\|^2.\notag
\end{align*}
Hence,
\begin{align*}
    &\bE_k \|p^{k+1} - p^\star\|^2+ \left(1+2\tau\mu_{H^*}\right)\bE_k\|q^{k+1} - q^\star\|_{\gamma,\tau}^2 + \kappa \gamma^2\bE_k\sigma_{k+1}^2 \\
    &\quad\leq \|p^{k} - p^\star\|^2 + \|q^{k} - q^\star\|_{\gamma,\tau}^2 -\frac{\gamma\mu}{(1+\gamma \lambda)^2}\|p^{k}-p^\star\|^2+ \kappa \gamma^2\left(1-\rho+\frac{\beta}{\kappa }\right)\sigma_k^2.\notag
\end{align*}
Thus, by setting $V^k$ as in  \eqref{eq:lyapunov-smooth} and $r$ as in \eqref{eq:rate-smooth}, we have $
    \bE_k V^{k+1} \leq r V^k
$.
\end{proof}

Thus, under smoothness and strong convexity assumptions, Theorem~\ref{th:Hsmooth:PD3O} implies  linear convergence of the dual variable $y^k$ to $y^\star$, with convergence rate given by $r$.
Since $\|x^k - x^\star\| \leq \|p^k - p^\star\|$, it also implies linear convergence of the variable $x^k$ to $x^\star$, with same  rate.

If $g^{k+1} = \nabla F(x^k)$, the Stochastic PD3O Algorithm reverts to the PD3O Algorithm and Theorem~\ref{th:Hsmooth:PD3O} provides a convergence rate similar to Theorem 3 in~\cite{yan18}. In this case, by taking $\kappa = 1$, we obtain
\begin{equation*}
  r = \max\left(1-\gamma\frac{\mu_F + 2\mu_R}{(1+\gamma \lambda)^2},\frac{1}{1+2\tau\mu_{H^*}}\right),
\end{equation*}
whereas Theorem 3 in~\cite{yan18} provides the rate
\begin{equation*}
\max\left(1-\gamma\frac{2(\mu_F+\mu_R) - \gamma\alpha\mu_F}{(1+\gamma \lambda)^2},\frac{1}{1+2\tau\mu_{H^*}}\right)
\end{equation*}
(the reader might not recognize the rate given in Theorem 3 of~\cite{yan18} because of some typos in Eqn.~39 of~\cite{yan18}).

\begin{theorem}[Linear convergence of the Stochastic PDDY Algorithm]
    \label{th:Hsmooth:PDDY}
    Suppose that Assumption~\ref{as:sto-grad} holds. Also, suppose that
     $H$ is $1/\mu_{H^*}$-smooth and $R$ is $\mu_R$-strongly convex, for some $\mu_R >0$ and $\mu_{H^*} >0$.
     Suppose that the parameters $\gamma>0$ and $\tau >0$ 
    satisfy $\gamma \leq 1/(\alpha+\kappa\delta)$, for some $\kappa > \beta/\rho$, 
   $\gamma\tau\|L\|^2 < 1$, and $\gamma^2 \leq \frac{\mu_{H^*}}{\|L\|^2 \mu_R}$.
   Define $\eta \eqdef 2\left(\mu_{H^*} -\gamma^2\|L\|^2\mu_R\right) \geq 0$ and, for every $k\in\mathbb{N}$, 
    \begin{equation}\label{eqV00}
        V^k \eqdef (1+\gamma\mu_R)\|p^{k} - p^\star\|^2 + (1+\tau\eta)\|y^{k} - y^\star\|_{\gamma,\tau}^2 + \kappa \gamma^2\sigma_{k}^2,
    \end{equation}
     and
     \begin{equation}\label{eqr00}
        r \eqdef \max\left(\frac{1}{1+\gamma\mu_R},1-\rho+\frac{\beta}{\kappa},\frac{1}{1+\tau\eta}\right)
    \end{equation}
    Then,  for every $k\in\mathbb{N}$,
    $
     \bE V^{k} \leq r^k V^0
    $.
\end{theorem}
\begin{proof}
We first use Lemma~\ref{lem:PDDY} along with the strong convexity of $R$ and $H^*$. Note that $y^{k} = q^{k+1}$. We have
\begin{align*}
    \bE_k \|v^{k+1} - v^\star\|_P^2 + \kappa \gamma^2\bE_k\sigma_{k+1}^2 \leq{}& \|v^k - v^\star\|_P^2 + \kappa \gamma^2\left(1-\rho+\frac{\beta}{\kappa }\right)\sigma_k^2\notag\\
    &-2\gamma\mu_{H^*}\bE_k \|q^{k+1} - q^\star\|^2 -2\gamma\mu_{R}\bE_k\|s^{k+1} - s^\star\|^2.
\end{align*}
Note that $s^{k+1} = p^{k+1} - \gamma L^* y^k$. Therefore, $s^{k+1} - s^\star = (p^{k+1} - p^\star) - \gamma L^*  (y^k - y^\star)$. Using Young's inequality $-\|a+b\|^2 \leq -\frac{1}{2}\|a\|^2 + \|b\|^2$, we have
$
-\bE_k\|s^{k+1} - s^\star\|^2 \leq -\frac{1}{2}\bE_k\|p^{k+1} - p^\star\|^2 + \gamma^2\|L\|^2 \bE_k\|q^{k+1} - q^\star\|^2
$. Hence, using $\tau \|q\|_{\gamma,\tau}^2 \leq \gamma \|q\|^2$,
\begin{align*}
    \bE_k \|v^{k+1} - v^\star\|_P^2 + \kappa \gamma^2\bE_k\sigma_{k+1}^2
    \leq{}& \|v^k - v^\star\|_P^2 + \kappa \gamma^2\left(1-\rho+\frac{\beta}{\kappa }\right)\sigma_k^2\\
    &-2\gamma\left(\mu_{H^*} -\gamma^2\|L\|^2\mu_R\right)\bE_k \|q^{k+1}- q^\star\|^2 \\
    &- \gamma\mu_{R}\bE_k\|p^{k+1} - p^\star\|^2\notag\\
    \leq{}& \|v^k - v^\star\|_P^2 + \kappa \gamma^2\left(1-\rho+\frac{\beta}{\kappa }\right)\sigma_k^2\\
    & -2\tau\bE_k \|q^{k+1}- q^\star\|_{\gamma,\tau}^2\left(\mu_{H^*} -\gamma^2\|L\|^2\mu_R\right) \\
    &- \gamma\mu_{R}\bE_k\|p^{k+1} - p^\star\|^2.\notag
\end{align*}
Set $\eta \eqdef 2\left(\mu_{H^*} -\gamma^2\|L\|^2\mu_R\right) \geq 0$. Then
\begin{align*}
    &(1+\gamma\mu_R)\bE_k \|p^{k+1} - p^\star\|^2 + (1+\tau\eta)\bE_k \|q^{k+1} - q^\star\|_{\gamma,\tau}^2 + \kappa \gamma^2\bE_k\sigma_{k+1}^2\notag\\
     &\quad\leq \|v^k - v^\star\|_P^2 + \kappa \gamma^2\left(1-\rho+\frac{\beta}{\kappa }\right)\sigma_k^2.
\end{align*}
Thus, by setting $V^k$ as in  \eqref{eqV00} and $r$ as in \eqref{eqr00}, we have $ 
    \bE_k V^{k+1} \leq r V^k.
$ 
\end{proof}

\section{PriLiCoSGD and Application to Decentralized Optimization}\label{apppridec}

   \begin{figure*}[t]
\begin{minipage}{.48\textwidth}\begin{algorithm}[H]
 \caption*{\textbf{PriLiCoSGD} (new)}
 \begin{algorithmic}[1]
    \STATE \textbf{Input:} $x^0\in \mathcal{X}$, $a^0\in \ran(W)$, $\gamma>0$, $\tau>0$
 	\FOR{$k = 0,1,2,\dots$}
	  \STATE $t^{k+1}=x^k - \gamma g^{k+1}$
	  \STATE $a^{k+1} = a^k + \tau W (t^{k+1}-\gamma a^k)-\tau c$
	  \STATE $x^{k+1} = t^{k+1}-\gamma a^{k+1}$
  	\ENDFOR
 \end{algorithmic}
\end{algorithm}\end{minipage}\ \ \ \ \   \begin{minipage}{.48\textwidth}
\begin{algorithm}[H]
    \caption*{
    \textbf{DESTROY} (new)}
 \begin{algorithmic}[1]
    \STATE \textbf{Input:} $x_i^0 \in \mathcal{X}$ and $a_i^0 \in \mathcal{X}$,   $\forall i \in V$, such that $\sum_{i\in V}a_i^0=0$, $\gamma>0$, $\tau>0$
 	\FOR{$k = 0,1,2,\dots$}
	\FOR{all $i \in V$ in parallel}
      \STATE $t_i^{k+1}=x_i^k - \gamma g_i^{k+1}$
	   \STATE $a_i^{k+1} = (1-\tau\gamma\widehat{W}_{i,i}) a_i^k +\tau\widehat{W}_{i,i} t_i^{k+1}$
	    \STATE ${}+ \tau \sum_{j\neq i:\{i,j\}\in V}\widehat{W}_{i,j} (t_j^{k+1}-\gamma a_j^k)$
	  \STATE $x_i^{k+1} = t_i^{k+1}-\gamma a_i^{k+1}$.
	  \ENDFOR
  		\ENDFOR
 \end{algorithmic}\end{algorithm}\end{minipage}\end{figure*}
 
 {In decentralized optimization, a network of computing agents aims at jointly minimizing an objective function by performing local computations and exchanging information along the edges~\cite{shi15,sca17,kov20,alg21}. It 
  is a particular case of linearly-constrained optimization, as detailed below. } 

First, let us set $W\eqdef L^*L$ and $c \eqdef L^*b$. Replacing the variable $y^k$ by the variable $a^k
\eqdef L^*y^k$ in LiCoSGD, we can write the algorithm using $W$ and $c$ instead of $L$, $L^*$ and $b$, with primal variables in $\mathcal{X}$ only. This yields the new algorithm PriLiCoSGD, shown above, to minimize $F(x)$ subject to $Wx=c$. The convergence results for LiCoSGD apply to PriLiCoSGD, with $(a^k)_{k\in\mathbb{N}}$ converging to $a^\star=-\nabla F(x^\star)$.

We can apply  PriLiCoSGD to decentralized optimization as follows.
Consider a connected undirected graph $G = (V,E)$, where $V = \{1,\ldots,N\}$ is the set of nodes and $E$ the set of edges. Consider a family $(f_i)_{i \in V}$ of $\mu$-strongly convex and $\nu$-smooth functions $f_i$, for some $\mu\geq 0$ and $\nu>0$. The problem is:
\begin{equation}
    \label{eq:decentralized}
    \min_{x \in \cX} \,\sum_{i \in V} f_i(x).
\end{equation}
Consider a gossip matrix of the graph $G$; that is, a $N \times N$ symmetric positive semidefinite matrix $\widehat{W} = (\widehat{W}_{i,j})_{i,j \in V}$, such that $\ker(\widehat{W}) = \Span([1\ \cdots\ 1]^\mathrm{T})$ and $\widehat{W}_{i,j} \neq 0$ if and only if $i=j$ or $\{i,j\} \in E$ is an edge of the graph. $\widehat{W}$ can be the Laplacian matrix of $G$, for instance.
Set $W \eqdef \widehat{W} \otimes I$, where $\otimes$ is the Kronecker product; 
then decentralized communication in the network $G$ is modeled by an application of the positive self-adjoint linear operator $W$ on $\cX^V$. Moreover, $W(x_1,\ldots,x_N) = 0$ if and only if $x_1 = \ldots = x_N$. 
Therefore, Problem~\eqref{eq:decentralized} is equivalent to the lifted problem
\begin{equation}
    \label{eq:dec2}
    \min_{\tilde{x} \in \cX^V} F(\tilde{x}) \quad \text{such that} \quad W \tilde{x} = 0, 
\end{equation}
where for every $\tilde{x}=(x_1,\ldots,x_N) \in \cX^V$, $F(\tilde{x}) = \sum_{i=1}^N f_i(x_i)$. Let us apply PriLiCoSGD to Problem~\eqref{eq:dec2}; we obtain the Decentralized Stochastic Optimization Algorithm (DESTROY).
 It 
generates the sequence  $(\tilde{x}^k)_{k\in\mathbb{N}}$, where $\tilde{x}^k = (x_1^k,\ldots,x_N^k) \in \cX^V$. The update of each $x_i^k$ consists in evaluating $g_i^{k+1}$, an estimate of $\nabla f_i(x_i^k)$ satisfying Assumption~\ref{as:sto-grad}, 
 and communication steps involving $x_j^k$, for every neighbor $j$ of $i$. 
For instance, the variance-reduced estimator $g_i^k$ can be the loopless SVRG estimator seen in Proposition~\ref{proplsvrg}, 
when $f_i$ is itself a  sum of functions, or a compressed version of $\nabla f_i$ \cite{sat20,bas20,li20a,xu21}.

As an application of the convergence results for LiCoSGD, we obtain the following results for DESTROY.
   Theorem \ref{th:pddy-cv} becomes:
    \begin{theorem}[Convergence of DESTROY, deterministic case $g_i^{k+1}=\nabla f_i(x_i^k)$]
    \label{th:destroy-cv}
    Suppose that $\gamma\in (0,2/\nu)$ and that $\tau\gamma\|\widehat{W}\|< 1$. Then in DESTROY,    each $(x_i^k)_{k\in\mathbb{N}}$ converges to the same
    solution $x^\star$ to the problem \eqref{eq:decentralized} and each $(a_i^k)_{k\in\mathbb{N}}$ converges to $a_i^\star=-\nabla f_i(x^\star)$.
\end{theorem}

Theorem \ref{th:cvx:lico} can be applied to the stochastic case, stating $\cO(1/k)$ convergence of the Lagrangian gap, by setting  $\mathcal{Y}=\mathcal{X}$ and $L=L^* = W^{1/2}$. 
Similarly, Theorem \ref{th:LV0} yields linear convergence of DESTROY in the strongly convex case $\mu>0$, with $L^*L$  replaced by $W$ and $\|L\|^2$ replaced by $\|W\|=\|\widehat{W}\|$. In particular, in the deterministic case,  
with $\gamma=1/\nu$ and $\tau\gamma =\aleph/\|W\|$ for some fixed $\aleph\in(0,1)$, 
$\varepsilon$-accuracy is reached after
   $\cO\Big(\max\big(\frac{\nu}{\mu},\frac{ \|W\|}{\omega(W)}\big)\log\big(\frac{1}{\varepsilon}\big)\Big)$ iterations. 
      This rate is better or equivalent to the one of recently proposed decentralized algorithms,    like EXTRA, DIGing, NIDS, NEXT, Harness, Exact Diffusion, see Table 1 of~\cite{xu2020distributed}, \cite[Theorem 1]{li20} and \cite{alg21}. With a stochastic gradient, the rate of our algorithm is also better than~\cite[Equation 99]{mokhtari2016dsa}.

In follow-up papers, the authors used Nesterov acceleration to propose accelerated versions of DESTROY~\cite{kov20} and PriLiCoSGD~\cite{sal21}.

 \bibliographystyle{spmpsci}
 
\bibliography{math,IEEEabrv,biblio}

\end{document}